%% file: main.tex
\documentclass[10pt, a4paper]{article}

\input{macros}

\usepackage[normalem]{ulem}
\begin{document}

% \title{Time-uniform Last-iterate SGD under a Sub-Gaussian Assumption}
\title{The Exact Time-Uniform Rate Frontier for Stochastic Gradient Descent on Smooth Convex Objectives}

\author{Ruijie Li \thanks{School of Data Science, Fudan University. Email: \texttt{24300720157@m.fudan.edu.cn}}
\and Kang Chen \thanks{Shanghai Center for Mathematical Sciences, Fudan University. Email: \texttt{kchen26@m.fudan.edu.cn}} \and Tianyu Wang \thanks{Shanghai Center for Mathematical Sciences, Fudan University. Email: \texttt{wangtianyu@fudan.edu.cn}}
}

\maketitle

\input{sections/abstract.tex}

\input{sections/intro-new}

\input{sections/preliminaries.tex}

\input{sections/sufficiency.tex}

\input{sections/necessity.tex}

\input{sections/conclusion.tex}

\input{sections/AI-statement}

\input{sections/Reproducibility-statement}

\bibliographystyle{unsrtnat}
\bibliography{references}

\appendix

\input{sections/Appendix-A-Deferred-proofs-from-section-2}

\input{sections/Appendix-B-Proofs-and-auxiliary-results-for-sufficiency}

\input{sections/Appendix-C-auxiliary-results-necessity}

\end{document}

%% file: macros.tex
\usepackage[left=24mm,right=24mm,top=22mm,bottom=22mm]{geometry}
\usepackage{amsmath,amsthm,amssymb}
\usepackage{tcolorbox}
\usepackage{graphicx}
\usepackage[utf8]{inputenc}
\usepackage{listings}
\usepackage{enumitem}
\usepackage[linesnumbered,ruled,vlined]{algorithm2e}
\usepackage{float}
\usepackage[numbers]{natbib}
\usepackage[colorlinks=true, linkcolor=blue, urlcolor=blue, citecolor=blue]{hyperref}

\newtheorem{theorem}{Theorem}[section]
\newtheorem{corollary}[theorem]{Corollary}
\newtheorem{lemma}[theorem]{Lemma}

\newtheorem{definition}[theorem]{Definition}

\newtheorem{remark}[theorem]{Remark}

%% file: sections/abstract.tex
\begin{abstract}
    % We study the time-uniform convergence of the raw iterate of standard stochastic gradient descent for unconstrained smooth convex objectives.
    % Under conditionally unbiased norm-sub-Gaussian gradient noise, we ask which overhead profiles can be achieved by a single deterministic stepsize schedule that is fixed before the
    % run, independent of both the stopping horizon and the confidence level. For every positive, eventually
    % nondecreasing sequence $h$ satisfying \(h(n) = o(\sqrt{n})\), we prove an exact characterization:
    % a dimension-free bound of order $h(n)/\sqrt{n}$, holding simultaneously for all $n$ with probability at least $1-\alpha$ and uniformly over the problem class, is achievable if and only if
    % \[
    %     \sum_{j = 1}^{\infty} \frac{1}{h(2^j)^2} < \infty.
    % \]
    % Thus, \((\log n)^{1/2+\varepsilon} / \sqrt{n}\) is achievable for every \(\varepsilon > 0\), whereas
    % \(\sqrt{\log n / n}\) is not. \textcolor{red}{The time-uniform convergence rate gets arbitrarily close to \(\sqrt{\log n / n}\) but never reaches it.} 
    % The constructive sufficiency result follows from a dyadic horizon-free schedule together with an additive conditional-restart inequality. The necessity counterpart applies to every deterministic nonnegative schedule and holds even for a one-dimensional analytic smooth convex objective with Gaussian noise.

    We study the time-uniform convergence of the raw iterate of standard stochastic gradient descent (SGD) for unconstrained smooth convex objectives. We prove that, under standard noise assumptions, the time-uniform convergence rate gets arbitrarily close to \(\sqrt{\log n / n}\) but never reaches it. More specifically, we prove that for every positive, eventually
    nondecreasing sequence $h$ satisfying \(h(n) = o(\sqrt{n})\), 
    a bound of order $h(n)/\sqrt{n}$, holding simultaneously for all $n$ with probability at least $1-\alpha$ and uniformly over the problem class, is achievable if and only if 
    \[ 
        \sum_{j = 1}^{\infty} \frac{1}{h(2^j)^2} < \infty. 
    \] 
    The constructive sufficiency result follows from a dyadic horizon-free schedule together with an additive conditional-restart inequality. The necessity counterpart applies to every deterministic nonnegative schedule and holds even for a one-dimensional analytic smooth convex objective with Gaussian noise.
    
\end{abstract}

%% file: sections/intro-new.tex
\section{Introduction}

Stochastic gradient descent (SGD), dating back to at least the stochastic approximation algorithms \citet{RobbinsMonro1951}, finds important applications in modern learning \citep[See e.g.][ for a review]{bottou2018optimization}. 
 In its simplest form, SGD iterates as 
 \begin{align}\label{eq:intro-sgd} 
    x_{t+1}=x_t-\eta_tG(x_t,\xi_t), 
\end{align}
where $G(x_t,\xi_t)$ is an estimate of $\nabla f(x_t)$, and $f$ is the objective function of interest. 

% Here is a polished rewrite with improved clarity, flow, and academic tone:

% ---

% Throughout the years, many researchers have contributed to this topic \citep[e.g.,][ and references therein]{lan2020first,liu2024revisiting,2506.23335,doi:10.1287/stsy.2022.0097, 10.1214/24-AAP2143}. 
Among many types of convergence guarantees, time-uniform guarantees are especially valuable for SGD, because they remain valid under data-dependent stopping. This property is crucial in practice, as stochastic algorithms are often terminated using dynamic, data-dependent rules. 
% A prominent example is early stopping in large-scale machine learning, where training is halted when validation performance stagnates or worsens—a widely used heuristic to mitigate overfitting \citep[e.g.,][]{prechelt2002early,dodge2020fine}.
From a theoretical perspective, such termination criteria can be formalized as stopping times. 
% Consider, for instance, a simple rule for an iterative stochastic algorithm:  
% \begin{align}  
%     \tau := \min \{ k \in \mathbb{N} : \| x_k - x_{k-1} \| \le \varepsilon \}, \label{eq:toy}  
% \end{align}  
% where $ x_k $ denotes the iterate at step $ k $, and $ \varepsilon > 0 $ is a fixed tolerance.
Crucially, the stopping time is inherently random, reflecting the stochastic nature of the iterates. A fixed-time guarantee at a predetermined iteration does not directly justify such a stopping rule. To obtain this type of convergence guarantee, we seek to establish bounds of the form: 
\begin{align}\label{eq:intro-time-uniform} \mathbb{P}\left(f(x_n)-f^*\leq b_{\alpha}(n),\; \forall n\geq1\right)\geq1-\alpha, \quad \forall \alpha \in (0,1), 
\end{align} 
where \(b_{\alpha}: \mathbb{N} \to \mathbb{R}_{+}\) is a boundary depending on \(\alpha\).
Unlike a collection of pointwise high-probability bounds, \eqref{eq:intro-time-uniform} is one event over the infinite trajectory and hence remains valid at every almost surely finite data-dependent stopping time. This viewpoint is closely related to confidence sequences in sequential inference \citep{darling1967confidence, lai1976confidence,HowardEtAl2021} in statistical inference. 
% , but our target is a deterministic envelope for each raw SGD iterate rather than an observable, trajectory-dependent certificate. 

We study this question for standard SGD on unconstrained smooth convex objectives over real Hilbert spaces under conditionally unbiased norm-sub-Gaussian noise. We ask which profiles \(h\) permit \eqref{eq:intro-time-uniform} with \(b_{\alpha}(n) = \mathcal{O}(h(n)/\sqrt n)\) under one deterministic horizon- and confidence-independent schedule, uniformly over the problem class and without dimension dependence.

\subsection{Contributions}\label{sec:intro-contributions}

    Our main result characterizes exactly every achievable overhead profile \(h\).

    % In this paper, we prove the exact order of $h$ in \eqref{eq:intro-time-uniform}, as stated in the following theorem. 

    \begin{theorem}[Exact profile frontier, informal]\label{thm:intro-frontier}
        Let $h:\mathbb{N}_+\to(0,\infty)$ be an eventually nondecreasing profile satisfying $h(n)=o(\sqrt{n})$. There exists a deterministic, confidence-independent, horizon-free stepsize schedule for the SGD recursion \eqref{eq:intro-sgd} on smooth convex objectives whose iterates satisfy \eqref{eq:intro-time-uniform} with \(b_{\alpha}(n) = \mathcal{O}(h(n) / \sqrt{n})\), if and only if 
        % simultaneously for all $n$ with probability at least $1-\alpha$ and uniformly over the problem class, if and only if
        \begin{align}\label{eq:intro-frontier}
            \sum_{j=1}^{\infty}\frac{1}{h(2^j)^2}<\infty. 
        \end{align}
         % Furthermore, there is no asymptotically smallest profile $h$. 
    \end{theorem}
    
    In particular, a time-uniform bound of order $\mathcal{O} ( ( \log n)^{(1+\varepsilon)/2} / \sqrt{n} )$ is achievable for every $\varepsilon>0$, whereas $\mathcal{O} \left( \sqrt{ \log n / n  } \right) $ is not. In addition, the time-uniform bound gets arbitrarily close to the order of $ \sqrt{\log n / n } $ through progressively slower iterated-logarithmic overheads, but never reaches it. 

    The formal version of our main result appears in Theorem~\ref{thm:main}.  This criterion is exact at the level of arbitrary slowly varying profiles, not merely a prescribed logarithmic power. 
    % It shows that ; admissible iterated-logarithm profiles approach the boundary still more closely. 
    % In Corollary \ref{cor:no-smallest-profile}, we further prove that there is no asymptotically smallest profile $h$ -- meaning that the time-uniform bound gets arbitrarily close to the order of $ \sqrt{ \frac{\log n}{n} } $ but never reaches it. The meaning of no asymptotically smallest profile is in Definition \ref{def:achievable-profile}.  
    Corollary \ref{cor:no-smallest-profile} establishes that no asymptotically smallest profile $h$ exists. %Specifically, while the time-uniform bound can approach the order of $\sqrt{\log n / n}$ arbitrarily closely, it never attains this rate. 
    Also, the Cauchy condensation test shows that \eqref{eq:intro-frontier} is equivalent to \(\sum_{n=1}^{\infty} (n h(n)^2)^{-1} < \infty\).
    A machine-checked Lean 4 formalization of the targeted mathematical results can be found in \url{https://github.com/Su-Zi-Zhan/Time-Uniform-SGD-Formalization}.
    
    % One readily checks that, for eventually nondecreasing $h$, condition \eqref{eq:intro-frontier} is equivalent to the integral test 
    % $$\int_M^\infty \frac{\mathrm{d}t}{t h(t)^2}<\infty, $$ 
    % where $M\geq 2$ is sufficiently large. Thus the dyadic base is inessential, and no pointwise smallest achievable overhead exists.

    % The sufficiency result is constructive: it gives one schedule, independent of both the stopping horizon and the confidence level, whose raw iterates satisfy a single time-uniform bound. The matching necessity result applies to every deterministic nonnegative schedule and already holds on one-dimensional analytic smooth convex instances with Gaussian noise. Hence the frontier is not caused by non-smoothness, high dimension, projection, or adversarial state-dependent noise.

\subsection{Related works}

    For non-smooth stochastic optimization, standard polynomial stepsizes incur logarithmic losses in raw-iterate convergence \citep{ShamirZhang2013}, and \citet{HarveyEtAl2019} showed that such losses are unavoidable for the schedules they considered while establishing corresponding high-probability guarantees. Horizon-dependent schedules can recover the optimal rates \citep{JainEtAl2019}, whereas modified methods such as FTRL-based momentum attain the optimal \(\mathcal O(T^{-1/2})\) expected last-iterate rate without knowing \(T\) \citep{LiLiuOrabona2022}.

    In deterministic non-smooth optimization, \citet{ZamaniGlineur2025} used a moving-comparator argument to derive exact worst-case last-iterate bounds and showed that no universal stepsize sequence attains the optimal guarantee at every horizon. Their moving-comparator argument terminalizes weighted one-step inequalities. \citet{LiuZhou2024} extended this mechanism to stochastic composite mirror descent, attaining the optimal fixed-horizon high-probability raw-iterate scale for smooth convex objectives under sub-Gaussian noise. We retain this terminalization mechanism but redesign the remaining-time weights for conditional restarts. Whereas their nonincreasing weights use a non-positive Bregman remainder to absorb directional martingale variation, ours separate the harmonic and confidence costs, replacing $\mathsf{H}_r(1+\log(1/\delta))$ by $\mathsf{H}_r+\log(1/\delta)$. This additive separation enables simultaneous control over infinitely many epochs and terminal positions.

    % In the deterministic non-smooth setting, \citet{ZamaniGlineur2025} derived exact worst-case last-iterate bounds for projected subgradient methods and a horizon-dependent linearly decaying schedule attaining the black-box lower bound, while ruling out a universal sequence that achieves this exact guarantee at every horizon. Their moving-comparator construction terminalizes a weighted family of one-step inequalities. \citet{LiuZhou2024} extended this mechanism to stochastic composite mirror descent and obtained the optimal fixed-horizon high-probability last-iterate scale for smooth convex objectives under sub-Gaussian noise; their schedule depends on the prescribed horizon, and its optimized tuning also depends on the confidence level. Our argument retains this terminalization mechanism but uses remaining-time weights tailored to conditional restarts. Unlike the nonincreasing auxiliary weights of \citet{LiuZhou2024}, whose non-positive Bregman remainder absorbs directional martingale variation, our weights separate the harmonic and confidence costs, yielding the additive dependence $\mathsf{H}_r+\log(1/\delta)$ instead of $\mathsf{H}_r(1+\log(1/\delta))$. This separation is essential for controlling infinitely many epochs and terminal positions.

    Recent results provide important time-uniform benchmarks. 
    For smooth convex objectives, \citet{FengJiangWangYing2023} proved an $\mathcal{O}((1+\log(1/\beta))\log k/\sqrt{k})$ anytime bound for a particular stochastic gradient descent with momentum (SGDM) scheme; later \citep{FengJiangWangYing2025} sharpened this to $(\log k)^{(1+\varepsilon)/2}/\sqrt{k}$, $\varepsilon\in(0,1/2)$, for a variant of SGDM. 
    Under strong convexity or related contractive structure, \citet{ChenFengWang2026} and \citet{PhamRinaldoSarkar2025} established sharp \(\mathcal O((\log\log k+\log(1/\beta))/k)\)-type uniform rates.
    % Under contractive geometry, \citet{ChenFengWang2026} established the tight uniform-in-time rate $\mathcal{O}((\log\log k+\log(1/\beta))/k)$ for smooth strongly convex SGD, with extensions to PL objectives and contractive stochastic approximation with additive noise. Using a quantitative Robbins--Siegmund framework, \citet{PhamRinaldoSarkar2025} obtained the same order for strongly convex and PL SGD under bounded noise, together with applications to Oja's algorithm and Robbins--Monro stochastic approximation. 
    Moreover, \citet{AolariteiJordan2025} derived anytime-valid weighted-suboptimality bounds for projected SGD on general convex domains. These results concern momentum, contractive objectives, or weighted certificates, rather than the raw iterate of standard SGD on the general smooth convex class.
    % For smooth convex objectives, \citet{FengJiangWangYing2023} proved an $\mathcal{O}((1+\log(1/\beta))\log k/\sqrt{k})$ anytime bound for a particular SGDM scheme. Their subsequent stopping-time analysis \citep{FengJiangWangYing2025} developed a large-deviation inequality for almost supermartingales and obtained the sharper family $(\log k)^{(1+\varepsilon)/2}/\sqrt{k}$, for every $\varepsilon\in(0,1/2)$, for a variant of SGDM. Under contractive geometry, \citet{ChenFengWang2026} established the uniform-in-time rate $\mathcal{O}((\log\log k+\log(1/\beta))/k)$ for smooth strongly convex SGD, with tight $\log\log k$ dependence at fixed confidence, and extended the result to PL objectives and contractive stochastic approximation with additive noise. \citet{PhamRinaldoSarkar2025} developed a quantitative Robbins--Siegmund framework for time-uniform concentration of recursive stochastic algorithms, obtaining an $\mathcal{O}((\log\log k+\log(1/\beta))/k)$ rate for strongly convex and PL SGD under bounded-noise conditions, together with applications to Oja's algorithm and Robbins--Monro stochastic approximation. In another direction, \citet{AolariteiJordan2025} constructed anytime-valid bounds for the weighted suboptimality of projected SGD on general convex domains; their fully observable confidence sequence requires a known uniform radius, as on a bounded domain. These results concern momentum schemes, contractive objective classes, or weighted observable certificates, rather than the raw last iterate of standard SGD on the general smooth convex class.

    In the non-smooth Lipschitz setting over bounded convex domains, \citet{KornowskiShamir2026} proved that no fixed infinite stepsize schedule can achieve an anytime raw-iterate rate $o((\log T)^{1/8}/\sqrt{T})$. To the best of our knowledge, no existing work has characterized the exact time-uniform overhead frontier for the raw last iterate of standard SGD over the general smooth convex class when one deterministic schedule must be independent of both the horizon and the confidence level.
    % What remains unresolved for standard SGD in the non-contractive smooth convex regime is the precise overhead forced jointly by the raw last iterate, a confidence-independent horizon-free schedule, and a single infinite-horizon event.

\subsection{Technical overview}\label{sec:intro-technical-overview}

    For sufficiency part, we use dyadic epochs \(N_j := 2^j\) with stepsize \(\gamma_j\asymp R / (\sigma\sqrt{N_j}h(N_j))\) where \(\sigma\) and \(R\) are constants.
    The stochastic energy of this epoch is $N_j\gamma_j^2\asymp R^2/(\sigma^2h(N_j)^2)$, so \eqref{eq:intro-frontier} is exactly the condition that the total squared-step energy be finite. The main analytical device is an additive conditional-restart inequality based on remaining-time weights. It controls every terminal position in an epoch while separating the harmonic and confidence terms, after which a summable allocation of failure probabilities yields one event for all $n$.
    
    To prove the necessity part of Theorem~\ref{thm:intro-frontier}, we consider a class of one-dimensional analytic quartic-flat convex objectives with Gaussian noise. Such objectives separate the two constraints on the schedule: sufficient cumulative steps are required for deterministic progress, while sufficiently small blockwise squared-step energy is needed to suppress stochastic excursions. Independent Gaussian block events convert these constraints into a deterministic Riccati-type recurrence. If \eqref{eq:intro-time-uniform} held while the series in \eqref{eq:intro-frontier} diverged, the recurrence would be forced to remain bounded and diverge simultaneously. Thus the two mechanisms identify the same reciprocal-square frontier. 

% To ensure this kind of convergence guarantee, we aim at prove bounds of the following form

 % Let $\mathcal H$ be a real separable Hilbert space. For minimizing the objective function $f:\mathcal H \to \mathbb R$, starting from an initial point $x_0 \in \mathcal H$, SGD generates 
    
    % where $G(x_t,\xi_t)$ is an estimate of $\nabla f(x_t)$. For general convex stochastic optimization, the canonical stochastic scale is $n^{-1/2}$ \citep{NemirovskiEtAl2009,AgarwalEtAl2012}. Averaging is the classical route to optimality, from asymptotic statistical efficiency to non-asymptotic general-convex guarantees \citep{PolyakJuditsky1992,NemirovskiEtAl2009}, whereas algorithms are commonly reported through the raw last iterate. This distinction is mathematically substantive.

%% file: sections/preliminaries.tex
\section{Preliminaries} \label{sec:preliminaries}

We use the standard separability convention for real Hilbert spaces. Every space in scope is 
finite-dimensional or isomorphic to \( \ell_2 \). Accordingly, let
\( \mathfrak{H}=\{ \mathbb{R}^d:d\in\mathbb{N}_{+} \}
\cup \{ \ell_2 \} \). The problem class is defined as follows.

\begin{definition} \label{def:problem-class}
    Fix \( L, R, \sigma > 0 \). Let $\mathfrak{P}_{\mathrm{cvx}}(L, R, \sigma)$ denote the class of problem instances $\mathcal{P}$ consisting of the following:
    % An admissible problem instance \( \mathcal{P} \in \mathfrak{P}_{\mathrm{cvx}}(L, R, \sigma) \)
    % consists of the following elements:
    \begin{enumerate}[label=(\roman*)]
        \item The ambient space \(\mathcal{H} \in \mathfrak{H}\) is selected before the run. We denote
        the corresponding inner product and norm by \( \langle \cdot, \cdot \rangle \) and
        \( \| \cdot \| \), respectively.
        \item The objective \( f: \mathcal{H} \to \mathbb{R} \) is convex, Fr\'echet differentiable,
        and \( L \)-smooth, with nonempty minimizer set \( X^* \). 
        \item The initial point $x_0\in\mathcal{H}$ satisfies $\operatorname{dist}(x_0,X^*)\leq R$. Fix $x^*\in X^*$ such that $\|x_0-x^*\|\leq R$, and write $f^*:=\min_{x\in\mathcal{H}}f(x)=f(x^*)$. 
        % \item The initial point \( x_0 \in \mathcal{H} \) and a minimizer \( x^* \in X^* \) satisfy
        % \( \|x_0 - x^* \| \leq R \). For brevity, we write \( f^* = f(x^*) \).
        \item Let \( \{ \xi_t \}_{t \geq 0} \) be fresh, independent random seeds, and let
        \( \mathcal{F}_{t} := \sigma (\xi_{0}, \dots, \xi_{t-1}) \) be the pre-query filtration.
        For every \(\mathcal{F}_t\)-measurable query \(x_t\), the oracle returns a stochastic gradient $G(x_t, \xi_t)$. We require this response to be \(\mathcal{F}_{t+1}\)-measurable, and define
        \(Z_t:=G(x_t,\xi_t)-\nabla f(x_t)\). The stochastic gradient oracle is conditionally unbiased
        and norm-sub-Gaussian in the following almost-sure sense:
        \[
            \mathbb{E}[Z_t\mid\mathcal{F}_t]=0, \quad
            \mathbb{E}\left[\left.
                \exp\left(\frac{\|Z_t\|^2}{\sigma^2}\right)
            \right|\mathcal{F}_t\right]\leq\mathrm{e}, \quad \forall t\geq0.
        \]
    \end{enumerate}
    We refer to any $\mathcal{P}\in\mathfrak{P}_{\mathrm{cvx}}(L, R, \sigma)$ as an admissible problem instance.
\end{definition}

For a deterministic infinite schedule \( \eta = \{\eta_{t}\}_{t \geq 0} \) with \( \eta_{t} \geq 0 \), consider the standard SGD \eqref{eq:intro-sgd} and define
\[
    % x_{t+1} = x_{t} - \eta_{t} G(x_{t}, \xi_{t}), \quad 
    \Delta_{t} := f(x_{t}) - f^*.
\]
The schedule is called horizon-free because the full infinite sequence is determined before the run; it is
not retuned for a terminal horizon. More precisely, the schedule is 
\( \eta_{t} = \eta_{t}(h, L, R, \sigma) \).
Our goal is to establish a time-uniform bound for standard SGD, which is given in the following definition. 

\begin{definition} \label{def:time-uniform-valid-boundary}
    % Let \( \mathfrak{P}_{\mathrm{cvx}}(L, R, \sigma) \) be the set of admissible problems defined in
    % Definition~\ref{def:problem-class} for some constants \( L, R, \sigma > 0 \). A stochastic gradient
    % descent method \( \mathcal{M} \) has a time-uniform convergence boundary \( b_{\alpha, \eta}(\cdot):
    % \mathbb{N}_{+} \to \mathbb{R}_{+} \) on \( \mathfrak{P}_{\mathrm{cvx}}(L, R, \sigma) \) with 
    % confidence level \( 1 - \alpha \), if 
    Fix $L,R,\sigma>0$, a confidence level $\alpha\in(0,1)$, and a deterministic infinite stepsize schedule $\eta=\{\eta_t\}_{t\geq0}$ with $\eta_t\geq0$. Let $\mathcal{M}_\eta$ denote the standard SGD method induced by $\eta$ on the problem class $\mathfrak{P}_{\mathrm{cvx}}(L,R,\sigma)$, and let $\mathbb{P}_{\mathcal{P}}^\eta$ denote the law of its trajectory on an instance $\mathcal{P}\in\mathfrak{P}_{\mathrm{cvx}}(L,R,\sigma)$. A deterministic function $b_{\alpha,\eta}:\mathbb{N}_+\to\mathbb{R}_+$ is called a time-uniform convergence boundary for $\mathcal{M}_\eta$ over $\mathfrak{P}_{\mathrm{cvx}}(L,R,\sigma)$ at confidence level $1-\alpha$ if and only if
    \[
        \inf_{\mathcal{P} \in \mathfrak{P}_{\mathrm{cvx}}(L, R, \sigma)} \mathbb{P}_{\mathcal{P}}^{\eta}
        \bigl(\Delta_{n} \leq b_{\alpha, \eta}(n),\; \forall n \geq 1\bigr) \geq 1 - \alpha.
    \]
\end{definition}

This probability concerns one joint event over the entire infinite run, and \( b_{\alpha, \eta} \) must
not depend on the realized sample path.

We now define the class of rate profiles considered in this paper:
\[
    \mathcal{H}_{0} := \{ h: \mathbb{N}_+ \to (0, \infty): h \text{ is eventually nondecreasing 
    and } h(n) = o(\sqrt{n}) \}.
\]
Here, ``eventually nondecreasing'' means that there exists \( n_{0} \in \mathbb{N}_+ \) such that
\( h(n) \) is nondecreasing for all \( n \geq n_{0} \). For any \( h \in \mathcal{H}_{0} \), 
it acts as a multiplicative overhead on the polynomial rate \( \mathcal{O}(1/\sqrt{n}) \), which
is the optimal fixed-horizon rate for SGD on general smooth convex objectives, since we need
\( h(n) \) to elevate a fixed-horizon guarantee to a strictly time-uniform guarantee. %With this terminology in place, we can define the class of achievable profiles.

% \begin{definition} \label{def:achievable-profile}
%     A profile \( h \in \mathcal{H}_{0} \) is called achievable if, for every $L, R, \sigma>0$, there exists one deterministic nonnegative infinite 
%     schedule \( \eta = \eta(h, L, R, \sigma) \), fixed before the run and chosen independently of the horizon, the problem 
%     instance, and the confidence level, such that the raw iterates of $\mathcal{M}_\eta$ satisfy the following property: for every \( \alpha \in (0, 1) \), there exists
%     a constant \( C_{h, \alpha, L, R, \sigma} < \infty \) such that
%     \[
%         \inf_{\mathcal{P} \in \mathfrak{P}_{\mathrm{cvx}}(L, R, \sigma)} \mathbb{P}_{\mathcal{P}}^{\eta}
%         \bigg(\Delta_{n} \leq C_{h, \alpha, L, R, \sigma} \min \bigg\{ 1, \frac{h(n)}{\sqrt{n}} \bigg\},\; 
%         \forall n \geq 1 \bigg)
%         \geq 1 - \alpha.
%     \]
%     We denote the set of all achievable profiles by \(\mathcal{H}_{\mathrm{ach}}\).
% \end{definition}

% \sout{The profile must be selected
% before the run. A different admissible profile may use a different schedule.}

With the problem class and operational constraints formally established, we state our main result. The
following theorem reveals that the boundary between achievable and unachievable overhead profiles is
exactly determined by a reciprocal-square series.

\begin{theorem}[Exact reciprocal-square profile frontier, formal] \label{thm:main}
    For every fixed \( L, R, \sigma > 0 \) and \(h \in \mathcal{H}_{0}\), the following statements are equivalent.
    \begin{enumerate}[label=(\roman*)]
        \item There exists a deterministic infinite schedule \(\eta = \eta(h, L, R, \sigma)\), independent of the horizon, problem instance and confidence level, such that for every \(\alpha \in (0, 1)\), there exists \(C_{h, \alpha, L, R, \sigma} < \infty\) satisfying
        \[
        \inf_{\mathcal{P} \in \mathfrak{P}_{\mathrm{cvx}}(L, R, \sigma)} \mathbb{P}_{\mathcal{P}}^{\eta}
         \bigg(\Delta_{n} \leq C_{h, \alpha, L, R, \sigma} \min \bigg\{ 1, \frac{h(n)}{\sqrt{n}} \bigg\},\; 
         \forall n \geq 1 \bigg)
         \geq 1 - \alpha.
        \]
        \item 
        \[
        \sum_{j=1}^{\infty} \frac{1}{h(2^j)^2} < \infty.
        \]
    \end{enumerate}
    We call profiles satisfying these equivalent conditions achievable, and denote their collection by \(\mathcal{H}_{\mathrm{ach}}\)
\end{theorem}

% In Sections~\ref{sec:sufficiency} and \ref{section:necessity}, we will
% prove Theorem~\ref{thm:main} in two parts, namely sufficiency and necessity.
% The sufficiency part establishes the upper bound by constructing an admissible schedule whenever the reciprocal-square series converges ((ii) \(\implies\) (i), see Theorem~\ref{thm:sufficiency}). The necessity part establishes the matching lower bound by showing that, when the series diverges, no deterministic nonnegative schedule can attain a time-uniform boundary of order $h(n)/\sqrt{n}$ ((i) \(\implies\) (ii), see Theorem~\ref{thm:necessity}). Together, these two directions establish the exact characterization of $\mathcal{H}_{\mathrm{ach}}$.

Sections~\ref{sec:sufficiency} and \ref{section:necessity} proves sufficiency ((ii) \(\implies\) (i), see Theorem~\ref{thm:sufficiency}) and necessity ((i) \(\implies\) (ii), see Theorem~\ref{thm:necessity}). Together, these two directions establish the exact characterization of $\mathcal{H}_{\mathrm{ach}}$.
% \begin{remark}
%     Although Theorem~\ref{thm:main} requires a single schedule independent of the confidence level, the necessity result below is proved in a stronger form: for each fixed $\alpha\in(0,1)$, the schedule is allowed to depend on $\alpha$ (see Theorem~\ref{thm:necessity}). Thus, allowing different deterministic horizon-free schedules at different confidence levels does not enlarge the class of achievable profiles.
% \end{remark}
By Theorem~\ref{thm:main}, we have the following corollary, whose proof is
deferred to Appendix~\ref{appendix:deferred-proofs-section2}.

\begin{corollary} \label{cor:no-smallest-profile}
    The set of achievable profiles \( \mathcal{H}_{\mathrm{ach}} \) has no asymptotically
    smallest element, i.e., if \( h \in \mathcal{H}_{\mathrm{ach}} \), then there exists
    \( g \in \mathcal{H}_{\mathrm{ach}} \) such that \( g(n) = o(h(n)) \). 
\end{corollary}

To illustrate the corollary, consider the following example. Let \( \ell_k(n) = \overbrace{\log \cdots 
\log}^{k \text{ times}} n \) for all sufficiently large $n$. Each profile below is understood to be extended to $\mathbb{N}_+$ by an arbitrary positive modification on a finite prefix. Then \( \sqrt{\ell_{1}(n)} \) is not achievable, but \( (\ell_{1}(n))^{1/2+\varepsilon} \)
is achievable for any \( \varepsilon > 0 \). Similarly, \( (\ell_1 \cdots \ell_k(n))^{1/2} \) is not 
achievable, but \( (\ell_1 \cdots \ell_{k-1}(n) \cdot \ell_{k}^{1+\varepsilon}(n))^{1/2} \) is achievable
for any \( \varepsilon > 0 \). Let \( h_{k}(n) = (\ell_1 \cdots \ell_k^{1+\varepsilon}(n))^{1/2} \). Then \( h_{k} \in \mathcal{H}_{\mathrm{ach}} \) and \( h_{k+1}(n) = o(h_{k}(n)) \). Hence, there is no asymptotically smallest achievable profile in the family of iterated logarithms, which is consistent with Corollary~\ref{cor:no-smallest-profile}. 
Also, this shows that the time-uniform convergence rate gets arbitrarily close to $ \mathcal{O} ( \sqrt{ \log n / n } ) $ but never reaches it. 

%% file: sections/sufficiency.tex
\section{Sufficiency of the reciprocal-square profile}
\label{sec:sufficiency}

In this section, we prove the constructive sufficiency, or upper-bound, direction of Theorem~\ref{thm:main}: every profile $h\in\mathcal{H}_0$ satisfying $\sum_{j\geq1}h(2^j)^{-2}<\infty$ belongs to $\mathcal{H}_{\mathrm{ach}}$. For arbitrary $L,R,\sigma>0$, we construct a deterministic infinite schedule $\eta(h,L,R,\sigma)$, independent of the horizon, the problem instance, and the confidence level, whose raw SGD iterates satisfy the $\mathcal{O}(h(n)/\sqrt{n})$ bound simultaneously for all $n$ on a single high-probability event.

\subsection{The horizon-free dyadic schedule}

For brevity and clarity, we employ a dyadic epoch-based schedule that
maintains a constant stepsize within each epoch to leverage martingale concentration while systematically decreasing
the stepsize across epochs to drive optimization error to zero.
Let \( N_{j} = 2^j \) and \(h_{j} = h(N_{j})\). Algorithm \ref{alg:schedule} completely specifies our schedule.

\begin{algorithm}
    \caption{Horizon-Free Dyadic Schedule for an Admissible Profile}
    \label{alg:schedule}
    \SetAlgoLined
    \KwIn{A profile \( h \in \mathcal{H}_{0} \) satisfying \( \sum_{j \geq 1} h(2^j)^{-2} < \infty \),
    and parameters \( L, R, \sigma > 0 \).}
    \KwOut{One deterministic infinite schedule \(\eta\) depending only on \( h, L, R, \sigma \).}
    % \( J_{h} := \inf \{ j \in \mathbb{N}_{+}: h \text{ is nondecreasing on } [N_{j}, +\infty) \} \).
    \( J_{h} := \inf \{ j \in \mathbb{N}_{+}: h(n+1)\ge h(n) \text{ for all } n\in\mathbb{N}_+ \text{ with } n\ge N_j \} \).

    \( c_{h} := \max \{ 1, (\sum_{j \geq J_{h}} h_{j}^{-2})^{1/2} \} \).

    \( j_{0} := \inf \{ j \geq J_{h}: h_{j} \geq 1, R/(\sigma c_{h} \sqrt{N_{j}} h_{j}) \leq 1/(2L) \} \).

    \For{\( t = 0, 1, 2, \dots \)}{
        \If{\( t < N_{j_{0}} \)}{
            \( \eta_{t} := 0 \).
        }
        \If{\( N_{j} \leq t < N_{j+1} \) for some \( j \geq j_{0} \)}{
            \( \eta_{t} = \gamma_{j} := \frac{R}{\sigma c_{h} \sqrt{ N_{j} } h_{j}} \).
        }
    }

    \Return{\( \eta = \{ \eta_{t} \}_{t \geq 0} \).}
\end{algorithm}

Algorithm~\ref{alg:schedule} is well defined. Indeed, $\sum_{j\geq J_h}h_j^{-2}<\infty$ implies $h_j\to\infty$, and hence $\sqrt{N_j}h_j\to\infty$. Thus the set defining $j_0$ is nonempty and $j_0<\infty$. Moreover, since $h_j$ is nondecreasing for $j\geq J_h$, the active stepsizes $\gamma_j=R/(\sigma c_h\sqrt{N_j}h_j)$ are nonincreasing.
Crucially, the total squared-step energy of this infinite
sequence is bounded above by \( R^2 / \sigma^2 \), as guaranteed by the convergence of the
reciprocal-square series. 

\begin{remark}
    A constant stepsize is adopted within each epoch primarily for simplicity of the sufficiency proof. Alternative stepsize schedules can also be employed.
\end{remark}

% \textit{Remark}: A constant stepsize is adopted within each epoch primarily for simplicity of the sufficiency proof. Alternative stepsize schedules can also be employed.

\subsection{Core tool: additive conditional restart}

The sufficiency proof repeatedly applies finite-block last-iterate estimates along the dyadic schedule. Since a block may start from an iterate determined by the preceding trajectory, and since the comparator may be \(\mathcal F_m\)-measurable, we need an estimate that remains valid conditionally on \(\mathcal F_m\). We refer to such an application as an analytical restart; the SGD trajectory itself is not reset.

A direct specialization of the high-probability last-iterate bound of \citet{LiuZhou2024} yields a stochastic term of the form \(\sigma^2 \gamma \mathsf{H}_{r} (1 + \log \frac{1}{\delta})\), where \(\mathsf{H}_{r} = \sum_{i=1}^{r} i^{-1}\) is the $r$-th harmonic number. Under repeated confidence allocation, this product incurs an additional logarithmic loss. Using the same barycentric last-iterate reduction, but with the remaining weights \(w_t=r-t+1\), we instead obtain the additive dependence
\(\sigma^2\eta\bigl(\mathsf{H}_r+\log(1/\delta)\bigr)\). The following theorem formalizes this conditional last-iterate estimate.

\begin{theorem}\label{thm:additive-conditional-restart}
    % Fix deterministic integers \( m \geq 0 \) and \( r \geq 1 \). 
    Fix $L,R,\sigma>0$, an instance $\mathcal{P}\in\mathfrak{P}_{\mathrm{cvx}}(L,R,\sigma)$, and a deterministic nonnegative stepsize schedule $\eta$. Let $m\in\mathbb{N}$ and $r\in\mathbb{N}_+$ be deterministic, and suppose that
    \( \eta_{m} = \cdots = 
    \eta_{m+r-1} = \gamma \), \( 0 < \gamma \leq 1/(2L) \). Then for every almost surely finite
    \( \mathcal{F}_{m} \)-measurable comparator \( y \) and every \( \delta \in (0, 1) \),
    \[
        \mathbb{P} \bigg(
            f(x_{m+r}) - f(y) \leq \frac{4D(y, x_{m})}{\gamma r} + 6\sigma^2\gamma \mathsf{H}_{r}
            + 9\sigma^2\gamma \log \frac{1}{\delta} \bigg| \mathcal{F}_{m}
        \bigg) \geq 1 - \delta \quad \text{almost surely,}
    \]
    where $D:\mathcal{H}\times\mathcal{H}\to[0, +\infty)$ is given by $D(a, b) = \|a-b\|^2/2$. 
\end{theorem}

See Appendix~\ref{appendix:additive-conditional-restart} for a complete proof. Theorem~\ref{thm:additive-conditional-restart} supplies the conditional finite-block estimate used in each restart. The time-uniform conclusion will follow by combining it with the global radius control (detailed in Appendix~\ref{appendix:global-radius}) and the dyadic schedule.

\subsection{Proof of sufficiency}

\begin{lemma} \label{lem:summability-epoch}
    If \( \{ h_{j} \}_{j \in \mathbb{N}} \) is positive, eventually nondecreasing, and
    \( \sum_{j=1}^{\infty} h_{j}^{-2} < \infty \), then \( j / h_{j}^{2} \to 0 \) as \( j \to \infty \).
\end{lemma}

\begin{proof}
    For every sufficiently large \( j \),
    \(
        \sum_{k = \lfloor j/2 \rfloor}^{j} h_k^{-2} \geq j h_j^{-2} / 2
    \).
    The left-hand side is a tail of a convergent series and tends to zero.
\end{proof}

\begin{theorem}[Sufficiency part of Theorem~\ref{thm:main}] \label{thm:sufficiency}
    Let $h\in\mathcal{H}_0$ satisfy $\sum_{j=1}^{\infty} h(2^j)^{-2} < \infty$. Then, for every $L,R,\sigma>0$, there exists a deterministic nonnegative infinite stepsize schedule $\eta=\eta(h,L,R,\sigma)$, independent of the terminal horizon, the problem instance, and the confidence level, such that, for every $\alpha\in(0,1)$, there exists a constant $C_{h,\alpha,L,R,\sigma}<\infty$ satisfying
    % Let \( h: [1, \infty) \to (0, \infty) \) be eventually nondecreasing and satisfy
    % \( \sum_{j=1}^{\infty} h(2^j)^{-2} < \infty \). Then there is one deterministic infinite schedule
    % \( \eta = \eta(h, L, R, \sigma) \) such that for every \( \alpha \in (0, 1) \), there exists a 
    % finite constant \( C_{h, \alpha, L, R, \sigma} \) for which
    \[
        \inf_{\mathcal{P} \in \mathfrak{P}_{\mathrm{cvx}}(L, R, \sigma)} \mathbb{P}_{\mathcal{P}}^{\eta}
        \bigg(\Delta_{n} \leq C_{h, \alpha, L, R, \sigma} \cdot\min \left\{ 1, \frac{h(n)}{\sqrt{n}} \right\}, \;
        \forall n \geq 1 \bigg)
        \geq 1 - \alpha.
    \]
\end{theorem}

The schedule is the one produced by Algorithm~\ref{alg:schedule}; it depends on \( h, L, R, \sigma \)
alone, and in particular not on \( \alpha \), the horizon, or the instance. The proof has three stages.
First, the finite total energy of the schedule confines the entire infinite trajectory to one fixed
ball. Epoch \( j \) contributes squared-step energy \( N_j \gamma_j^2 = \mathcal{O}(h_j^{-2})\),
so reciprocal-square summability makes the total energy finite. This yields a single high-probability
radius bound for the entire infinite trajectory and makes every active stepsize admissible for
Theorem~\ref{thm:additive-conditional-restart} (Lemma~\ref{lem:schedule-calibration} and
Lemma~\ref{lem:global-radius}).

Second, inside each dyadic epoch, two applications of
Theorem~\ref{thm:additive-conditional-restart} with different comparators control every terminal
position. For late terminal positions, enough steps have elapsed that \( x^* \) can be used directly
as the comparator in Theorem~\ref{thm:additive-conditional-restart}. But for early terminal positions,
since the bound scales like \(1/r\) with the number of steps \(r\), the penalty is too large to reach
the target bound. Instead, we use the starting point \(x_{N_j}\) as the comparator, which eliminates
the penalty term. The missing bound on \( \Delta_{N_j} \) is exactly the late-half bound inherited
from epoch \( j-1 \). Because adjacent dyadic scales differ only by a constant factor, this handoff
propagates one constant across all epochs (Lemmas~\ref{lem:within-epoch} and \ref{lem:epoch-propagation}).

Finally, uniform control of all \( 2^j \) terminal positions in epoch \( j \) requires a summable
confidence allocation, for which both the harmonic term and \( \log(1 / \delta) \) are \( \mathcal{O}(j) \).
The additive dependence in Theorem~\ref{thm:additive-conditional-restart} therefore makes the
stochastic cost, relative to the target \(h_j / \sqrt{N_j}\), of order \(j / h_j^2\), which vanishes
by Lemma~\ref{lem:summability-epoch}. Here the additive structure is crucial: a multiplicative
dependence would instead produce \(j^2 / h_j^2\) and lose the frontier. 

\begin{proof}[Proof of Theorem~\ref{thm:sufficiency}]
    Fix \( \alpha \in (0, 1) \) and an arbitrary
    \( \mathcal{P} \in \mathfrak{P}_{\mathrm{cvx}}(L, R, \sigma) \).
    By Lemma~\ref{lem:schedule-calibration}, the schedule of Algorithm~\ref{alg:schedule} satisfies
    \( \gamma_j \leq 1/(2L) \) for every \( j \geq j_0 \) and \( Q_{\infty} \leq R^2 / \sigma^2 \).
    Hence Lemma~\ref{lem:global-radius} applies at level \( \alpha/2 \), and the global radius event
    \( \mathcal{G} \) of \eqref{eq:global-radius-event} satisfies
    \( \mathbb{P}(\mathcal{G}) \geq 1 - \alpha/2 \). With the allocation
    \( \delta_{j,r} = \alpha / 2^{2(j+1)} \) of \eqref{eq:confidence-allocation}, the restart events
    \( \mathcal{A}_{j,r} \) satisfy \( \mathbb{P}(\mathcal{A}_{j,r}) \geq 1 - \delta_{j,r} \), and
    \[
        \sum_{j = j_0}^{\infty} \sum_{r = 1}^{N_j} \delta_{j, r}
        = \sum_{j = j_0}^{\infty} N_j \frac{\alpha}{2^{2(j+1)}}
        \leq \frac{\alpha}{2}.
    \]
    A union bound therefore gives \( \mathbb{P}(\mathcal{E}) \geq 1 - \alpha \) for
    \[
        \mathcal{E} := \mathcal{G} \cap \bigcap_{j \geq j_0} \bigcap_{1 \leq r \leq N_j}
        \mathcal{A}_{j,r}.
    \]
    On \( \mathcal{E} \), Lemma~\ref{lem:sufficiency-prefix} controls every
    \( 1 \leq t < 3 N_{j_0} / 2 \) with the constant \( C_{\mathrm{prefix}} \), and
    Lemma~\ref{lem:epoch-propagation} controls every \( t \geq 3 N_{j_0} / 2 \) with the constant
    \( C_{\mathrm{latter}} \). Finally, \( L \)-smoothness at the minimizer and
    \eqref{eq:global-radius-event} give \( \Delta_t \leq L B_{\alpha/2} \leq \frac{81}{4} L R^2
    \log(4/\alpha) \) on \( \mathcal{G} \), so the boundary may also be capped at one. Setting
    \[
        C_{h, \alpha, L, R, \sigma} := \max \left\{ C_{\mathrm{prefix}}, C_{\mathrm{latter}},
        \frac{81}{4} L R^2 \log \frac{4}{\alpha} \right\}
    \]
    yields \( \Delta_t \leq C_{h, \alpha, L, R, \sigma} \min \{ 1, h(t) / \sqrt{t} \} \)
    simultaneously for all \( t \geq 1 \) on \( \mathcal{E} \). Since \( \mathcal{P} \) was
    arbitrary and \( \eta \) does not depend on it, the infimum over
    \( \mathfrak{P}_{\mathrm{cvx}}(L, R, \sigma) \) obeys the same bound.
\end{proof}

%% file: sections/necessity.tex
\section{Necessity of reciprocal-square summability}
\label{section:necessity}

We prove that reciprocal-square summability is necessary for a deterministic schedule to attain a
time-uniform \( h(n)/\sqrt{n} \) boundary. It is enough to work with the uncapped envelope \(\varepsilon_{\alpha}(n) \leq B \cdot h(n) / \sqrt{n}\),
where \( B \) is a constant. This is because the above envelope is weaker than the capped boundary in
Theorem~\ref{thm:main}.

\begin{theorem}[Necessity part of Theorem~\ref{thm:main}] \label{thm:necessity}
    Fix $L,R,\sigma>0$ and $\alpha\in(0,1)$. Let $h\in\mathcal{H}_0$ and let $\eta=\{\eta_t\}_{t\geq0}$ be any deterministic nonnegative infinite stepsize schedule, fixed before the run and independent of the terminal horizon and the problem instance. The schedule may depend on $h,L,R,\sigma$, and $\alpha$. Suppose that $\mathcal{M}_\eta$ admits a time-uniform convergence boundary $b_{\alpha,\eta}$ over $\mathfrak{P}_{\mathrm{cvx}}(L,R,\sigma)$ at confidence level $1-\alpha$. If there exist constants $B>0$ and $n_0\in\mathbb{N}_+$ such that
    % Let \( \eta = \{ \eta_t \}_{t \geq 0} \) be a deterministic schedule, \( \alpha \in (0, 1) \) 
    % be a failure probability, and \( h: [1, \infty) \to (0, \infty) \) be a positive, eventually
    % nondecreasing function such that \( h(n) = o(\sqrt{n}) \). If a valid level-\( \alpha \)
    % time-uniform boundary \(b_{\alpha, \eta}\) satisfies
    \[
        b_{\alpha, \eta}(n) \leq B \cdot \frac{h(n)}{\sqrt{n}}, \quad \forall n \geq n_0,
    \]
    % for some constant \( B > 0 \) and all \( n \geq n_0 \), 
    then \(\sum_{j=1}^{\infty} h(2^j)^{-2} < \infty\).
\end{theorem}

\begin{remark}
    Allowing $\eta$ to depend on $\alpha$ only strengthens the necessity statement: the conclusion therefore applies directly to the confidence-independent schedules required in the definition of achievability.
\end{remark}

Throughout the proof, the schedule, the profile, and all constants in Theorem~\ref{thm:necessity}
are fixed. We write \( S_n := \sum_{t<n} \eta_t \), \( Q_n := \sum_{t<n} \eta_t^2 \), and
\( \varepsilon_{\alpha}(n) := b_{\alpha, \eta}(n) \). All lemmas and corollaries in the remainder of this section are understood to hold under the conditions of Theorem~\ref{thm:necessity}. 

\subsection{Fixed-horizon analysis}

Since time-uniform boundaries imply fixed-horizon boundaries:
\[
    \mathbb{P}\bigl(\Delta_n \leq \varepsilon_{\alpha}(n)\bigr) \geq \mathbb{P}\bigl(\Delta_t \leq \varepsilon_{\alpha}(t),\;
    \forall t \geq 1\bigr) \geq 1 - \alpha,
\]
we can apply the fixed-horizon analysis to obtain some necessary results, which are summarized in the
following two lemmas. The first comes from the amount of deterministic progress required to 
reduce the initial error.

\begin{lemma} \label{lem:deterministic-progress-obstruction}
    For any \( n \in \mathbb{N}_{+} \), 
    \[
        \varepsilon_{\alpha}(n) \geq \frac{R^2}{16} \min \left\{ L, \frac{1}{S_n} \right\},
    \]
    where \( 1/S_n := +\infty \) when \( S_n = 0 \).
\end{lemma}

Lemma~\ref{lem:deterministic-progress-obstruction} shows that a sufficiently large total stepsize
\(S_n\) is necessary to reduce the initial error. The second restriction comes from the stochastic energy
injected by the oracle noise.

\begin{lemma} \label{lem:stochastic-energy-obstruction}
    Suppose \( S_n \geq 1/L \). Let \( z_{\alpha} := \Phi^{-1}(1 - \alpha/2) \), where $\Phi$ is the cumulative distribution function of the standard normal distribution. Then
    \[
        \varepsilon_{\alpha}(n) \geq \frac{z_{\alpha}^2 \sigma^2}{64} \cdot \frac{Q_n}{S_n}.
    \]
\end{lemma}

Lemma~\ref{lem:stochastic-energy-obstruction} establishes a trade-off constraint: the error bound is
governed by the ratio \( Q_n / S_n \). Consequently, accommodating a large \( Q_n \) requires a correspondingly large $S_n$ in order to achieve a sufficiently small error bound. The proofs of
Lemmas~\ref{lem:deterministic-progress-obstruction} and \ref{lem:stochastic-energy-obstruction}
are deferred to Appendix~\ref{appendixsub:fixed-time-analysis}.

Now let \( N_j = 2^j \), \(h_j = h(2^j)\), and \(\mathcal{Q}_j = Q_{N_j}\). We define \( p_j :=
S_{N_j} / \sqrt{N_j} \). Then we have the following corollary, whose proof can also be seen in Appendix~\ref{appendixsub:fixed-time-analysis}:
\begin{corollary} \label{cor:step-lower-bound}
    For sufficiently large \( j \in \mathbb{N}_{+} \), we have
    \[
        p_j \geq \frac{z_{\alpha}^2 \sigma^2}{64 B} \cdot\frac{\mathcal{Q}_j}{h_j}.
    \]
\end{corollary}

These lemmas provide necessary results, but they cannot place infinitely many epoch constraints 
on one common event. The next subsection uses one fixed instance throughout the infinite run; this
is where the joint time-uniform guarantee is essential.

\subsection{A fixed flat-active witness and the dyadic epoch ceiling}
\label{subsec:necessity-witness}

Define
\[
    \phi(y) := \frac{L}{4} \left(\sqrt{R^2 + y^2} - R\right)^2
\]
It is straightforward to verify that \( \phi \) is convex, \( L \)-smooth, and non-PL. We can further verify
that \( \phi \) is quartic-flat at the minimizer by Taylor expansion:
\[
    \phi(y) = \frac{L}{16 R^2} y^4 + \mathcal{O}(y^6), \quad \text{as } y \to 0.
\]
Moreover,
\[
    \phi'(y) = \frac{Ly^3}{2\sqrt{R^2+y^2}(\sqrt{R^2+y^2}+R)}, \quad |\phi'(y)|\leq\frac{L|y|^3}{4R^2}.
\]
Let \( \vartheta := (1 - \mathrm{e}^{-2})/2 \), \( \nu^2 := \vartheta \sigma^2 \).
The recursion is initialized at the minimizer and uses independent \( \zeta_t \sim \mathcal{N}(0, \nu^2) \)
as the oracle noise:
\[
    Y_{t+1} = Y_{t} - \eta_t (\phi'(Y_t) + \zeta_t), \quad Y_0 = 0.
\]
The Gaussian oracle satisfies the noise assumption with equality: \( \mathbb{E}[\exp(\zeta_t^2/\sigma^2)]
= (1 - 2 \nu^2 / \sigma^2)^{-1/2} = \mathrm{e} \). Therefore, it is a valid problem instance defined in
Definition~\ref{def:problem-class}. By the assumption of Theorem~\ref{thm:necessity}, the all-time 
success event \( \mathcal{C} := \{ \phi(Y_n) \leq \varepsilon_{\alpha}(n), \forall n \geq 1 \} \) satisfies
\( \mathbb{P}(\mathcal{C}) \geq 1 - \alpha \), where \( \varepsilon_{\alpha}(n) \leq B h(n)/\sqrt{n} \) for all sufficiently large $n$.

Validity on this instance first implies that \( \eta_t \to 0 \) as \( t \to \infty \) (see 
Lemma~\ref{lem:deminishing-stepsize} in Appendix~\ref{appendixsub:preliminaries-Gaussian}).

Now we introduce a few quantities that summarize the behavior of the stepsize schedule and the boundary
on each dyadic epoch. First, over the interval \( [N_j, N_{j+1}] \), we consider the worst-case rescaling
of the overhead function and set
\[
    \hat{h}_j := \max_{\substack{N_j \leq n \leq N_{j+1} \\ n \in \mathbb{N}_{+}}} h(n) \sqrt{\frac{N_j}
    {n}}.
\]
This quantity captures the largest effective value of \( h(n) / \sqrt{n} \) within the \( j \)-th epoch,
expressed at the left endpoint scale \( N_j \).
Moreover, there exists a threshold \( j_h \) such that, for every \( j > i \geq j_h \), eventually 
monotonicity gives \( h_j \leq \hat{h}_j \leq h_{j+1} \) and \( \hat{h}_i \leq h_{i+1} \leq h_j \). 
Now we express the corresponding boundary level at epoch \( j \) as \( \bar{\varepsilon}_j := B \hat{h}_j / 
\sqrt{N_j} \). On the event \( \mathcal{C} \), all iterates in the \( j \)-th epoch satisfy
\( \phi(Y_n) \leq \bar{\varepsilon}_j \). This motivates defining
\[
    \beta_j^2 := \sup \{ |y|^2: \phi(y) \leq \bar{\varepsilon}_j \} = r^2(\bar{\varepsilon}_j), \quad
    r^2(u) := 4R \sqrt{\frac{u}{L}} + \frac{4u}{L}.
\]

The following lemma establishes the core time-uniform property of our lower bound.
\begin{lemma} \label{lem:dyadic-epoch-ceiling}
    Let \( \gamma_0 = (\log 2) / 4 \) and \( C_b = 800 / (\vartheta \gamma_0) \). Then for every 
    sufficiently large \( j \), we have
    \[
        H_j := \sum_{t=N_j}^{N_{j+1} - 1} \eta_t \leq C_b \frac{B}{\sigma^2} \frac{\hat{h}_j \sqrt{N_j}}{j}. 
    \]
\end{lemma}

The proof is based on a block anti-concentration argument. On the common success event, all iterates in
epoch \( j \) remain in a shrinking neighborhood of the minimizer, where the quartic-flat objective
produces only a small accumulated drift. If the epoch stepsize mass \(H_j\) were too large, the
epoch could be partitioned into many disjoint blocks, each of which would force an independent Gaussian
increment to lie in a prescribed small interval. Gaussian anti-concentration then shows that
all these constraints cannot hold with probability at least \( 1 - \alpha \).
Balancing the number of available blocks against the corresponding Gaussian tail scale yields the factor
\( 1/j \).
The proof is deferred to Appendices~\ref{appendixsub:preliminaries-Gaussian} and \ref{appendixsub:greedy-step-mass}.

\subsection{Reciprocal-square summability}

We now combine the fixed-horizon analysis with the dyadic epoch ceiling, yielding the following two lemmas. In 
Corollary~\ref{cor:step-lower-bound}, we established a lower bound for \( p_j \). The following lemma
provides an upper bound for \( p_j \), and together they yield an upper bound for \( \mathcal{Q}_j \).

\begin{lemma} \label{lem:step-upper-bound}
    There exist constants \( C_p \) and \(C_q\) such that for sufficiently large \( j \), we have
    \[
        p_j \leq C_p \frac{B}{\sigma^2} \frac{h_j}{j}, \quad
        \mathcal{Q}_j \leq C_q \frac{B^2}{\sigma^4} \frac{h_j^2}{j}.
    \]
\end{lemma}

The following nonlinear lower recurrence for \(\mathcal{Q}_j\) complements the upper bounds above. This recurrence is the key mechanism that connects the dyadic energy growth to the reciprocal-square series.

\begin{lemma} \label{lem:energy-amplification}
    Let \( m \) be sufficiently large. Then for all \( J > m \),
    \[
        \mathcal{Q}_J \geq (2 - \sqrt{2}) \mathcal{Q}_m + (\sqrt{2} - 1)^2 \frac{z_{\alpha}^4 \sigma^4}
        {2^{12} B^2} \sum_{j=m+1}^{J-1} \frac{\mathcal{Q}_j^2}{h_j^2}.
    \]
\end{lemma}
The proofs of Lemmas~\ref{lem:step-upper-bound} and \ref{lem:energy-amplification} are deferred to Appendix~\ref{appendixsub:dyadic-sequence-estimates}.
Now we are ready to prove Theorem~\ref{thm:necessity}.

\begin{proof}[Proof of Theorem~\ref{thm:necessity}]
    We use contradiction. Suppose \( \sum_{j=1}^{\infty} h_j^{-2} = \infty \).

    We choose \( m \) sufficiently large that all preceding eventual bounds hold. Let \( T_{m+1} :=
    (2 - \sqrt{2}) \mathcal{Q}_m  \). By Lemma~\ref{lem:deterministic-progress-obstruction}, $S_n \to +\infty$. Hence, for sufficiently large $m$, $\mathcal{Q}_m >0$, and therefore $T_{m+1}>0$. By Lemma~\ref{lem:energy-amplification}, \( T_{m+1} \leq
    \mathcal{Q}_{m+1} \). We recursively define
    \[
        T_{J+1} := T_{J} + \kappa \frac{T_J^2}{h_J^2}, \quad \forall J > m, \quad 0 < \kappa \leq
        (\sqrt{2} - 1)^2 \frac{z_{\alpha}^4 \sigma^4}{2^{12} B^2}.
    \]
    Then by Lemma~\ref{lem:energy-amplification} and induction, we have \( \mathcal{Q}_J \geq T_J \)
    for all \( J > m \). By Lemma~\ref{lem:step-upper-bound}, we have \( \kappa T_J / h_J^2 \leq \kappa C_q
    B^2 \sigma^{-4} / J \leq 1 \) for every sufficiently large \( J \). Therefore, we have
    \[
        \frac{1}{T_J} - \frac{1}{T_{J+1}} = \frac{T_{J+1} - T_J}{T_J T_{J+1}} = 
        \frac{\kappa T_J^2 / h_J^2}{T_J(T_J + \kappa T_J^2 / h_J^2)} = \frac{\kappa / h_J^2}{1 + \kappa
        T_J / h_J^2} \geq \frac{\kappa}{2 h_J^2}.
    \]
    Fix an index \( J_1 \) from which this inequality holds. Summing from \(J_1\) to \( \infty \) gives
    \[
        \frac{1}{T_{J_1}} \geq \sum_{j=J_1}^{\infty} \frac{1}{T_j} - \frac{1}{T_{j+1}} \geq
        \sum_{j=J_1}^{\infty} \frac{\kappa}{2 h_j^2} = \infty,
    \]
    which is a contradiction. Therefore, we must have \( \sum_{j=1}^{\infty} h_j^{-2} < \infty \).
\end{proof}

%% file: sections/conclusion.tex
\section{Conclusion}\label{sec:conclusion}

We have characterized the exact time-uniform convergence frontier of standard SGD on general smooth convex objectives: achievable overhead profiles are exactly those satisfying reciprocal-square summability \(\sum_{j=1}^{\infty} h(2^j)^{-2} < \infty\). The resulting picture is governed by reciprocal-square summability: rather than admitting a single optimal overhead, the achievable class has no asymptotically smallest element and can approach the \(\sqrt{\log n / n}\) arbitrarily closely without attaining it. Our necessity result shows that this obstruction already arises for a one-dimensional analytic smooth convex objective with Gaussian noise, even when the schedule may depend on the confidence level. Thus, the frontier is intrinsic to time-uniform raw-iterate control in the general noncontractive setting, rather than an artifact of nonsmoothness, projection, high dimension, or confidence-independent tuning. A natural next question is whether the same frontier persists for adaptive or randomized stepsize rules.

% The sufficiency result follows from a dyadic schedule of finite squared-step energy, an additive conditional-restart inequality, and a summable confidence allocation. For the necessity part, a one-dimensional quartic-flat Gaussian witness and block anti-concentration impose epochwise constraints on the stepsize mass, leading to a Riccati-type recurrence whose compatibility with the dyadic energy upper bound forces reciprocal-square summability. Since the construction uses analytic smooth convex objectives and Gaussian, possibly degenerate, noise, the obstruction is intrinsic rather than an artifact of non-smoothness, projection, high dimension, or state-dependent noise.

% The necessity argument in fact holds for a one-dimensional analytic smooth convex objective with Gaussian noise, at each fixed confidence level, even allows the deterministic schedule to depend on that confidence level. Thus, the obstruction is intrinsic to the general noncontractive setting rather than an artifact of non-smoothness, projection, high dimension, or confidence-independent tuning.

% Natural next questions are whether the same frontier persists for adaptive or randomized stepsizes, composite or non-Euclidean methods, and heavy-tailed noise, as well as whether observable time-uniform certificates with sharper finite-sample constants can be obtained.

%% file: sections/AI-statement.tex
\section*{AI use statement}

In this work, we used GPT-5.6 Sol to assist in developing and refining proof strategies for the necessity direction, including suggesting intermediate arguments and technical derivations. The mathematical claims were formulated by the authors, and all AI-assisted proof arguments were independently checked and manually verified by the authors. We did not use generative AI tools for generating the proof in the sufficiency direction. Specifically, GPT-6-Astra was used to assist in translating the mathematical development into Lean 4 and constructing the supplementary formalization. The resulting formalization was checked by the Lean kernel and independently audited for theorem correspondence, dependencies, and unintended axioms. GPT-5.6 Sol and Gemini 3.1 Pro was used to polish the writing of the introduction and related work sections. We take responsibility for the final content of this work, including text, claims, or artifacts produced with the aid of generative AI.

%% file: sections/Reproducibility-statement.tex
\section*{Reproducibility statement}

Complete proofs of all mathematical claims are provided in the appendices. We additionally provide an anonymized Lean 4 formalization of the targeted mathematical results as supplementary material, together with theorem-correspondence and proof-integrity audits, as well as instructions for rebuilding the formalization.

%% file: sections/Appendix-A-Deferred-proofs-from-section-2.tex
\section{Deferred proofs from Section~\ref{sec:preliminaries}}
\label{appendix:deferred-proofs-section2}
\begin{proof}[Proof of Corollary~\ref{cor:no-smallest-profile}]
    Let \( h \in \mathcal{H}_{\mathrm{ach}} \). By Theorem~\ref{thm:main}, we have
    \[
        \sum_{j=1}^{\infty} \frac{1}{h(2^j)^2} < \infty.
    \]
    Let \( h_{j} = h(2^j) \) and \(a_{j} = h_{j}^{-2}\).
    Then \(\sum_{j \geq 1} a_j < \infty\). In particular, \(a_j \to 0\) and \(h_j \to \infty\). Since \(h\) is eventually nondecreasing, there exists \(J_* \geq 1\) such that \(h\) is nondecreasing on \([2^{J_*}, \infty)\). Also, by the convergence of the series, we can choose a strictly increasing sequence of integers \(J_* \leq J_1 < J_2 < \cdots\) such that for every \(k \geq 1\), \(\sum_{j = J_k}^{\infty} a_j < 2^{-k} k^{-2}\). Now we define a positive sequence \(\{c_k\}_{k \geq 1}\) by
    \[
        c_j:= \begin{cases} 1, & 1 \le j < J_1 \\
            k, & J_k \le j < J_{k+1},\quad k \ge 1
        \end{cases}
    \]
    Then \(c_j \to \infty\). Moreover,
    \[
        \sum_{j=1}^{\infty} c_j a_j = \sum_{j=1}^{J_1-1} a_j + \sum_{k=1}^{\infty} k \sum_{j=J_k}^{J_{k+1}-1} a_j \leq \sum_{j=1}^{J_1-1} a_j + \sum_{k=1}^{\infty} k \sum_{j=J_k}^{\infty} a_j < \sum_{j=1}^{J_1-1} a_j + 
        \sum_{k=1}^{\infty} \frac{2^{-k}}{k} < \infty.
    \]
    Now define \(r_j := 1/\sqrt{c_j a_j} = h_j / \sqrt{c_j}\) and \(g_j := \max_{1 \leq i \leq j} r_i\) for \(j \geq 1\). Then \(\{g_j\}_{j \geq 1}\) is positive and nondecreasing. Since \(g_j \geq r_j\), \(g_j^{-2} \leq r_j^{-2} = c_j a_j\). Therefore,
    \[
        \sum_{j=1}^{\infty} \frac{1}{g_j^2} \leq \sum_{j=1}^{\infty} c_j a_j < \infty.
    \]
    We next prove \(g_j = o(h_j)\). For any \(\varepsilon > 0\), there exists an integer \(K \in \mathbb{N}_{+}\) such that \(K^{-1/2} < \varepsilon\). For every \(i \geq J_K\), we have \(c_i \geq K\). Since \(J_K \geq J_*\), for every \(J_K \leq i \leq j\) we have \(h_i \leq h_j\). Hence,
    \[
        r_i = \frac{h_i}{\sqrt{c_i}} \leq \frac{h_j}{\sqrt{K}} < \varepsilon h_j.
    \]
    Now for the finitely many indices \(i < J_K\), define \(M_K := \max_{1 \leq i < J_K} r_i\) with the convention \(M_K := 0\) if \(J_K = 1\). Since \(h_j \to \infty\), there exists \(J' \geq J_K\) such that \(M_K < \varepsilon h_j\) for any \(j \geq J'\). Thus, for every \(j \geq J'\), 
    \[
        g_j = \max_{1 \leq i \leq j} r_i \leq \varepsilon h_j,
    \]
    which implies \(g_j / h_j \to 0\) as \(j \to \infty\). Finally, we define \(g: \mathbb{N}_+ \to (0, \infty)\) by
    \[
        g(n) := g_{\max \{ 1, \lfloor \log_2 n \rfloor \}}.
    \]
    Now we only need to prove \(g(n) = o(h(n))\). For sufficiently large \(n\), let \(j = \lfloor \log_2 n \rfloor\). Then \(2^j \leq n < 2^{j+1}\), and eventual monotonicity of \(h\) yields \(h(n) \geq h(2^j) = h_j\). Therefore,
    \[
        0 \leq \frac{g(n)}{h(n)} = \frac{g_j}{h(n)} \leq \frac{g_j}{h_j} \to 0.
    \]
    Hence \(g(n) = o(h(n))\). Since \(h(n) = o(\sqrt{n})\), it follows that \(g(n) = o(\sqrt{n})\) and therefore \(g \in \mathcal{H}_0\). Moreover,
    \[
        \sum_{j=1}^{\infty} \frac{1}{g(2^j)^2} = \sum_{j=1}^{\infty} \frac{1}{g_j^2} < \infty.
    \]
    By Theorem \ref{thm:main}, \(g \in \mathcal H_{\mathrm{ach}}\). Since \(g(n) = o(h(n))\), every \(h \in \mathcal H_{\mathrm{ach}}\) admits another \(g \in \mathcal H_{\mathrm{ach}}\) that is asymptotically strictly smaller. Therefore, \(\mathcal H_{\mathrm{ach}}\) has no asymptotically smallest element.
\end{proof}

%% file: sections/Appendix-B-Proofs-and-auxiliary-results-for-sufficiency.tex
\section{Proofs and auxiliary results for sufficiency}

\input{sections/Appendix-B-1-conditional-restart}

\input{sections/Appendix-B-2-global-radius}

\input{sections/Appendix-B-3-sufficiency-proof}

%% file: sections/Appendix-B-1-conditional-restart.tex
\subsection{Additive conditional restart theorem: proof and supporting lemmas}
\label{appendix:additive-conditional-restart}

Let \( y \in \mathcal{H} \) be an arbitrary comparator. We define the squared-distance function
\( D: \mathcal{H} \times \mathcal{H} \to \mathbb{R}_{+} \cup \{ 0 \} \) by \( D(a, b) = \|a-b\|^2/2 \).

\subsubsection{Local smooth geometry}

Our analysis is based on two geometric lemmas.

\begin{lemma} \label{lem:smooth-convex-interpolation}
    For any \( x, u, z \in \mathcal{H} \),
    \[
        f(u) - f(z) \leq \langle \nabla f(x), u - z \rangle + \frac{L}{2} \|u - x\|^2.
    \]
\end{lemma}

\begin{proof}
    By the \( L \)-smoothness of \( f \), we have
    \[
        f(u) \leq f(x) + \langle \nabla f(x), u - x \rangle + \frac{L}{2} \|u - x\|^2.
    \]
    By the convexity of \( f \), we have
    \[
        f(x) - f(z) \leq \langle \nabla f(x), x - z \rangle.
    \]
    Adding these two inequalities yields
    \[
        f(u) - f(z) \leq \langle \nabla f(x), u - z \rangle + \frac{L}{2} \|u - x\|^2.
    \]
\end{proof}

\begin{lemma} \label{lem:one-step-inequality}
    Let \( x^+ = x - \eta (\nabla f(x) + \zeta) \) with \( 0 < \eta \leq 1 / (2L) \). Then for any
    \( z \in \mathcal{H} \),
    \[
        f(x^+) - f(z) \leq \frac{D(z, x) - D(z, x^+)}{\eta} + \langle \zeta, z - x \rangle + 
        \eta \|\zeta\|^2.
    \]
\end{lemma}

\begin{proof}
    Applying Lemma~\ref{lem:smooth-convex-interpolation} with \( u = x^+ \) gives
    \begin{align*}
        f(x^+) - f(z) &\leq \langle \nabla f(x), x^+ - z \rangle + \frac{L \eta^2}{2} 
        \|\nabla f(x) + \zeta\|^2 \\
        &= \frac{\langle \eta(\nabla f(x) + \zeta), x^+ - z \rangle}{\eta} - \langle \zeta, x^+ - z \rangle
        + \frac{L \eta^2}{2} \|\nabla f(x) + \zeta\|^2 \\
    \end{align*}
    By the identity
    \[
        \|a + b\|^2 = \|a\|^2 + 2 \langle a, b \rangle + \|b\|^2,
    \]
    we have
    \[
        \langle \eta(\nabla f(x) + \zeta), x^+ - z \rangle = -\frac{1}{2} \eta^2 \|\nabla f(x) + \zeta\|^2 
        - \frac{1}{2} \|x^+ - z\|^2 + \frac{1}{2} \|x - z\|^2.
    \]
    Hence, we have
    \[
        f(x^+) - f(z) \leq \frac{D(x, z) - D(x^+, z)}{\eta} - \langle \zeta, x^+ - z \rangle 
        + \frac{\eta}{2}(L\eta - 1) \|\nabla f(x) + \zeta\|^2
    \]
    Notice that
    \[
        \eta \langle \zeta, \nabla f(x) \rangle - \frac{\eta}{4} \|\nabla f(x) + \zeta\|^2 
        = -\frac{\eta}{4} \|\nabla f(x) - \zeta\|^2 \leq 0.
    \]
    Therefore, we have
    \[
        f(x^+) - f(z) \leq \frac{D(x, z) - D(x^+, z)}{\eta} + \langle \zeta, z - x \rangle 
        + \eta \|\zeta\|^2.
    \]
\end{proof}

Let \( \rho \in [0, 1] \) and set \( z = (1 - \rho) x + \rho q \). Since \( D(z, x) = \rho^2 D(q, x)
\leq \rho D(q, x) \), Lemma~\ref{lem:one-step-inequality} gives
\begin{equation*}
    f(x^+) - f(z) \leq \rho \langle \zeta, q - x \rangle + \frac{1}{\eta} (\rho D(q, x) - D(z, x^+))
    + \eta \|\zeta\|^2.
\end{equation*}

\subsubsection{Terminal weights and barycentric identities}

In the following, consider the iteration
\[
    X^{t+1} = X^{t} - \eta_{t} (\nabla f(X^{t}) + Z^{t}), \quad 1 \leq t \leq T, 0 < \eta_{t} \leq
    \frac{1}{2L}.
\]
for some \( T \geq 1 \).

Our objective in this section is to bound the global error \( f(X^{T+1}) - f(y) \) evaluated at an
arbitrary comparator \( y \in \mathcal{H} \). Since the one-step inequality in Lemma~\ref{lem:one-step-inequality}
bounds the single-step error \( f(X^{t+1}) - f(z^t) \) relative to a reference
point \( z^{t} \), we need to construct a linear combination of these single-step errors that
telescopes exactly to the global error.

Let \( a_{t} > 0 \) denote the linear-combination coefficient at time \(t\). We consider the corresponding
weighted sum of single-step errors. To relate \( f(z^{t}) \) to the iterates \( f(X^{t}) \), we define
\( z^{t} \) via a barycentric recursion. Let \( \{ v_{t} \}_{t=1}^{T} \) be a positive,
nondecreasing sequence of parameters with \( 0 < v_{0} = v_{1} \) and \(v_{T} = 1\), and define
\[
    z^{t} := \bigg( 1 - \frac{v_{t-1}}{v_{t}} \bigg) X^{t} + \frac{v_{t-1}}{v_{t}} z^{t-1}, \quad
    z^{0} := y.
\]
By applying convexity to this recursion, we have
\[
    v_{t} f(z^{t}) \leq (v_t - v_{t-1}) f(X^{t}) + v_{t-1} f(z^{t-1}) \leq v_{0} f(y) + \sum_{s=1}^{t}
    (v_{s} - v_{s-1}) f(X^{s}).
\]
Substituting this inequality into the linear combination of single-step errors, we have
\begin{align*}
    \sum_{t=1}^{T} a_{t} (f(X^{t+1}) - f(z^{t})) &\geq \sum_{t=1}^{T} a_{t} f(X^{t+1}) - \sum_{t=1}^{T}
    \frac{a_{t}}{v_{t}} \bigg( v_{0} f(y) + \sum_{s=1}^{t} (v_{s} - v_{s-1}) f(X^{s}) \bigg) \\
    &= \sum_{t=2}^{T} \bigg( a_{t-1} - (v_{t} - v_{t-1}) \sum_{k=t}^{T} \frac{a_{k}}{v_{k}} \bigg) f(X^t) \\
    &\quad - v_{0} f(y) \sum_{t=1}^{T} \frac{a_{t}}{v_{t}} + a_{T} f(X^{T+1}). \\
\end{align*}
Let \( a_{t-1} = (v_{t} - v_{t-1}) \sum_{k=t}^{T} \frac{a_{k}}{v_{k}} \). Then the first term vanishes.
Let \( b_{t} = \frac{a_{t}}{v_{t}} \) and \( W_{t} = \sum_{k=t}^{T} b_{k} \). Then we have
\( a_{t-1} = (v_{t} - v_{t-1}) W_{t} \). On the other hand,
\( a_{t-1} = b_{t-1} v_{t-1} \). Hence, we have the following identity:
\[
    b_{t-1} v_{t-1} = (v_{t} - v_{t-1}) W_{t}, \quad 2 \leq t \leq T.
\]
Also note that \( b_{t-1} = W_{t-1} - W_{t} \). Substituting this into the previous identity, we have
\[
    v_{t-1} W_{t-1} = v_{t} W_{t} \implies v_{t} W_{t} = v_{T} W_{T} = a_{T}, \quad 2 \leq t \leq T.
\]
Since \( v_{0} = v_{1} \), we have \( v_{0} W_{1} = v_{1} W_{1} = a_{T} \). Substituting this
back into the previous inequality, we have
\[
    \sum_{t=1}^{T} a_{t} (f(X^{t+1}) - f(z^{t})) \geq a_{T} (f(X^{T+1}) - f(y)).
\]
Now let \( b_{t} = w_{t} \eta_{t} \), where \( w_{t} \) is a nonincreasing, nonnegative weight sequence.
Then \( a_{t} = w_{t} \eta_{t} v_{t} \). Let \( a_{T} = C \). Then we have the following lemma.

\begin{lemma} \label{lem:deterministic-master-inequality}
    The following inequality holds:
    \[
        C(f(X^{T+1}) - f(y)) \leq w_{1} v_{0} D(y, X^{1}) + \sum_{t=1}^{T} c_{t} \|Z^{t}\|^2
        + \sum_{t=1}^{T} \beta_{t} \langle Z^{t}, d_{t} \rangle + \mathcal{V}_{T},
    \]
    where \( c_t = w_{t} \eta_{t}^2 v_{t} \), \( \beta_{t} = w_{t} \eta_{t} v_{t-1} \), \( d_{t} = z^{t-1}
    - X^{t} \), and
    \[
        \mathcal{V}_{T} = \sum_{t=2}^{T} (w_{t} - w_{t-1}) v_{t-1} D(z^{t-1}, X^{t}) \leq 0.
    \]
\end{lemma}

\begin{proof}
    Applying Lemma~\ref{lem:one-step-inequality} to each term in the summation \( \sum_{t=1}^{T} a_{t} 
    (f(X^{t+1}) - f(z^{t})) \), we have
    \begin{align*}
        &\quad\sum_{t=1}^T a_t\bigl(f(X^{t+1})-f(z^t)\bigr) \\
        &\le
        \sum_{t=1}^T w_t v_{t-1} D(z^{t-1},X^t)
        -\sum_{t=1}^T w_t v_t D(z^t,X^{t+1}) 
        +\sum_{t=1}^T c_t\lVert Z^t\rVert^2
        +\sum_{t=1}^T \beta_t\langle Z^t,d_t\rangle \\
        &=
        w_1 v_0 D(y,X^1)
        +\sum_{t=2}^T (w_t-w_{t-1})v_{t-1}D(z^{t-1},X^t) \\
        &\quad
        -w_Tv_TD(z^T,X^{T+1})
        +\sum_{t=1}^T c_t\lVert Z^t\rVert^2
        +\sum_{t=1}^T \beta_t\langle Z^t,d_t\rangle \\
        &\le
        w_1v_0D(y,X^1)
        +\sum_{t=1}^T c_t\lVert Z^t\rVert^2
        +\sum_{t=1}^T \beta_t\langle Z^t,d_t\rangle
        +V_T .
    \end{align*}
    where the last inequality follows from the fact that \( -w_{T} v_{T} D(z^T, X^{T+1}) \leq 0 \) and
    \[
        \mathcal{V}_{T} = \sum_{t=2}^{T} (w_{t} - w_{t-1}) v_{t-1} D(z^{t-1}, X^t).
    \]
    Since \( w_{t} \) is nonincreasing, \( \mathcal{V}_{T} \leq 0 \).
\end{proof}

\subsubsection{Proof of the additive conditional restart theorem}

Let \( w_{t} = T - t + 1 \). Consider the constant stepsize \( \eta_{t} = \gamma \) for \( 1 \leq 
t \leq T \). Then \( b_{t} = w_{t} \eta_{t} = \gamma w_{t} \) and \( W_{t} = \sum_{k=t}^{T} b_{k}
= \gamma w_{t} (w_{t} + 1) / 2 \). We also have \( C = a_{T} = w_{T} \gamma = \gamma \). Then by
the identity \( v_{t} W_{t} = C \), we have
\[
    v_{t} = \frac{C}{W_{t}} = \frac{2}{ w_{t} (w_{t} + 1)}.
\]
Substituting these expressions into Lemma~\ref{lem:deterministic-master-inequality} gives
\[
    f(X^{T+1}) - f(y) \leq \frac{2 D(y, X^1)}{\gamma(T+1)} + \sum_{t=1}^{T} w_{t} \gamma v_{t} \|Z^{t}\|^2
    + \sum_{t=1}^{T} w_{t} v_{t-1} \langle Z^{t}, d_{t} \rangle - \sum_{t=2}^{T} \frac{v_{t-1}}{2\gamma}
    \|d_{t}\|^2.
\]
To bound the error, we need the following two lemmas.

\begin{lemma} \label{lem:norm-martingale}
    For any \( \delta \in (0, 1) \), the following inequality holds:
    \[
        \mathbb{P} \bigg( \sum_{t=1}^{T} w_{t} \gamma v_{t} \|Z^{t}\|^2 \leq 2 \sigma^2 \gamma
        \mathsf{H}_{T} + \sigma^2 \gamma \log \frac{1}{\delta} \bigg) \geq 1 - \delta,
    \]
    where \( \mathsf{H}_{T} = \sum_{i=1}^{T} \frac{1}{i} \).
\end{lemma}

\begin{proof}
    Let \(a_* = \max_{1 \leq i \leq T} \{ w_{i} \gamma v_{i} \}\). Let \( \{ \mathcal{F}_{t} \}_{t=1}^{T} \)
    be a filtration such that \( Z^{t} \) is measurable with respect to \( \mathcal{F}_{t} \). 
    Let
    \[
        N_{t} = \exp \bigg( \sum_{i=1}^{t} \frac{a_{i}}{a_{*}} \bigg( \frac{\|Z^{i}\|^2}{\sigma^2} - 1 \bigg)
         \bigg), \quad a_{i} = w_{i} \gamma v_{i}.
    \]
    By the sub-Gaussian property of \(Z^{t}\), we have
    \[
        \mathbb{E} \bigg[ \exp \bigg( \frac{\|Z^{t}\|^2}{\sigma^2} - 1 \bigg) \bigg| \mathcal{F}_{t-1} \bigg] 
        \leq 1.
    \]
    Since \(a_{t} / a_{*} \leq 1\), \( x^{a_{t} / a_{*}} \) is concave for \( x \geq 0 \). By Jensen's
    inequality, we have
    \[
        \mathbb{E} \bigg[ \exp \bigg( \frac{a_{t}}{a_{*}} \bigg( \frac{\|Z^{t}\|^2}{\sigma^2} - 1 \bigg) 
        \bigg) \bigg| \mathcal{F}_{t-1} \bigg] 
        \leq \bigg\{\mathbb{E} \bigg[ \exp \bigg( \frac{\|Z^{t}\|^2}{\sigma^2} - 1 \bigg) 
        \bigg| \mathcal{F}_{t-1} \bigg] \bigg\}^{a_{t} / a_{*}}
        \leq 1.
    \]
    Since
    \[
        \frac{N_{t+1}}{N_{t}} = \exp \bigg( \frac{a_{t+1}}{a_{*}} \bigg( \frac{\|Z^{t+1}\|^2}{\sigma^2} - 1
         \bigg)  \bigg),
    \]
    we have
    \[
        \mathbb{E}(N_{t+1} / N_{t} \mid \mathcal{F}_{t}) \leq 1 \implies 
        \mathbb{E}(N_{t+1} \mid \mathcal{F}_{t}) \leq N_{t}.
    \]
    Therefore, \( \{N_{t}\} \) is a supermartingale. By Ville's inequality, we have
    \[
        \mathbb{P}(N_{T} \leq 1 / \delta) \geq 1 - \delta, \quad \forall \delta \in (0, 1).
    \]
    The event \( N_{T} \leq 1 / \delta \) implies
    \[
        \sum_{t=1}^{T} w_{t} \gamma v_{t} \|Z^{t}\|^2 \leq 2 \sigma^2 \gamma
        \mathsf{H}_{T} + \sigma^2 \gamma \log \frac{1}{\delta},
    \]
    where \( \mathsf{H}_{T} = \sum_{i=1}^{T} \frac{1}{i} \).
    Therefore, we have
    \[
        \mathbb{P} \bigg( \sum_{t=1}^{T} w_{t} \gamma v_{t} \|Z^{t}\|^2 \leq 2 \sigma^2 \gamma
        \mathsf{H}_{T} + \sigma^2 \gamma \log \frac{1}{\delta} \bigg) \geq 1 - \delta,
    \]
\end{proof}

\begin{lemma} \label{lem:inner-product-martingale}
    For any \( \delta \in (0, 1) \), the following inequality holds:
    \[
        \mathbb{P} \bigg( \sum_{t=1}^{T} w_{t} v_{t-1} \langle Z^{t}, d_{t} \rangle - \sum_{t=2}^{T} 
        \frac{v_{t-1}}{2\gamma} \|d_{t}\|^2 \leq \frac{2 D(y, X^1)}{\gamma (T+1)^2} + 8
        \sigma^2 \gamma \log \frac{1}{\delta} \bigg) \geq 1 - \delta.
    \]
\end{lemma}

\begin{proof}
    Let \( \mathcal{F}_{t} \) be a filtration generated by \( Z^{1}, \ldots, Z^{t} \).
    The sub-Gaussian property of \( Z^{t} \) implies that for any \( \lambda > 0 \),
    \[
        \mathbb{E} [ \exp(\langle Z^{t}, \lambda w_{t} v_{t-1} d_{t} \rangle) \mid \mathcal{F}_{t-1} ] \leq 
        \exp(2 \sigma^2 \lambda^2 w_{t}^2 v_{t-1}^2 \|d_{t}\|^2). 
    \]
    Hence, we can construct a supermartingale \( M_{t} \) as follows:
    \[
        M_{t} = \exp \bigg( \lambda \sum_{i=1}^{t} w_{i} v_{i-1} \langle Z^{i}, d_{i} \rangle - 
        2 \sigma^2 \lambda^2 \sum_{i=1}^{t} w_{i}^2 v_{i-1}^2 \|d_{i}\|^2 \bigg)
    \]
    It is straightforward to verify that \( M_{t} \) is a supermartingale. By Ville's inequality, we have
    \[
        \mathbb{P}(M_{T} \leq 1 / \delta) \geq 1 - \delta, \quad \forall \delta \in (0, 1).
    \]
    The event \( M_{T} \leq 1 / \delta \) implies
    \[
        \sum_{t=1}^{T} w_{t} v_{t-1} \langle Z^{t}, d_{t} \rangle \leq 2 \sigma^2 \lambda \sum_{t=1}^{T}
        w_{t}^2 v_{t-1}^2 \|d_{t}\|^2 + \frac{\log(1 / \delta)}{\lambda}.
    \]
    Note that
    \begin{align*}
        2 \sigma^2 \lambda \sum_{t=1}^{T} w_{t}^2 v_{t-1}^2 \|d_{t}\|^2 &= 2 \sigma^2 \lambda w_{1}^2 v_{0}^2
        \|d_{1}\|^2 + 2 \sigma^2 \lambda \sum_{t=2}^{T} w_{t}^2 v_{t-1}^2 \|d_{t}\|^2 \\
        &= 2 \sigma^2 \lambda w_{1}^2 v_{1}^2 \|d_{1}\|^2 + 2 \sigma^2 \lambda \sum_{t=2}^{T} w_{t}^2 
        v_{t-1}^2 \|d_{t}\|^2 \\
        &= \frac{8 \sigma^2 \lambda}{(T + 1)^2} \|d_{1}\|^2 + 2 \sigma^2 \lambda \sum_{t=2}^{T} w_{t}^2 
        v_{t-1}^2 \|d_{t}\|^2.
    \end{align*}
    Observe that
    \[
        \frac{2 \gamma(w_{t}^2 v_{t-1}^2)}{v_{t-1}} = \frac{4 \gamma w_{t}^2}{(w_{t} + 1)(w_{t} + 2)} 
        \leq 4 \gamma.
    \]
    Hence, we have
    \[
        2 \sigma^2 \lambda \sum_{t=2}^{T} w_{t}^2 v_{t-1}^2 \|d_{t}\|^2 \leq 8 \sigma^2 \lambda \gamma
        \sum_{t=2}^{T} \frac{v_{t-1}}{2 \gamma} \|d_{t}\|^2.
    \]
    Taking \( \lambda = 1 / (8 \sigma^2 \gamma) \), we have
    \[
        \sum_{t=1}^{T} w_{t} v_{t-1} \langle Z^{t}, d_{t} \rangle - \sum_{t=2}^{T} \frac{v_{t-1}}{2\gamma}
        \|d_{t}\|^2 \leq \frac{2 D(y, X^1)}{\gamma (T+1)^2} + 8 \sigma^2 \gamma \log \frac{1}{\delta}.
    \]
    This holds with probability at least \( 1 - \delta \), or equivalently,
    \[
        \mathbb{P} \bigg( \sum_{t=1}^{T} w_{t} v_{t-1} \langle Z^{t}, d_{t} \rangle - \sum_{t=2}^{T} 
        \frac{v_{t-1}}{2\gamma} \|d_{t}\|^2 \leq \frac{2 D(y, X^1)}{\gamma (T+1)^2} + 8
        \sigma^2 \gamma \log \frac{1}{\delta} \bigg) \geq 1 - \delta.
    \]
\end{proof}

With the master inequality and the two martingale bounds established, we are ready to prove 
Theorem~\ref{thm:additive-conditional-restart}. We first prove the following lemma, which is a direct
consequence of Lemmas~\ref{lem:deterministic-master-inequality}, \ref{lem:norm-martingale}, and
\ref{lem:inner-product-martingale}.

\begin{lemma} \label{lem:terminal-inequality}
    The following inequality holds with probability at least \( 1 - \delta \):
    \[
        f(X^{T+1}) - f(y) \leq \frac{4 D(y, X^{1})}{\gamma T} + 6 \sigma^2 \gamma \mathsf{H}_{T} +
        9 \sigma^2 \gamma \log \frac{1}{\delta}.
    \]
\end{lemma}

\begin{proof}
    We apply Lemma~\ref{lem:norm-martingale} with failure probability \( p \delta \) and 
    Lemma~\ref{lem:inner-product-martingale} with failure probability \( (1 - p) \delta \) for some
    \( p \in (0, 1) \). By the union bound, with probability at least \( 1 - \delta \),
    \begin{align*}
        f(X^{T+1}) - f(y) &\leq \frac{2D(y, X^{1})}{\gamma(T+1)} \bigg( 1 + \frac{1}{T+1} \bigg)
        + 2 \sigma^2 \gamma \mathsf{H}_{T} + 9 \sigma^2 \gamma \log \frac{1}{\delta} 
        - \sigma^2 \gamma (\log p + 8 \log (1 - p)) \\
        &\leq \frac{4D(y, X^{1})}{\gamma T} + 2\sigma^2 \gamma \mathsf{H}_{T} + 9 \sigma^2 \gamma
        \log \frac{1}{\delta} + \sigma^2 \gamma \bigg(\log 9 + 8 \log \frac{9}{8}\bigg) \\
        &\leq \frac{4D(y, X^{1})}{\gamma T} + 6\sigma^2 \gamma \mathsf{H}_{T} + 9 \sigma^2 \gamma
        \log \frac{1}{\delta},
    \end{align*}
    In the second inequality, we take \(p = 1/9\) to minimize the term \(-\log p - 8 \log (1 - p)\). Its minimum has absolute value at most \(4\), so the last inequality follows. Therefore,
    \[
        \mathbb{P} \bigg( f(X^{T+1}) - f(y) \leq \frac{4 D(y, X^{1})}{\gamma T} + 6 \sigma^2 \gamma
        \mathsf{H}_{T} + 9 \sigma^2 \gamma \log \frac{1}{\delta} \bigg) \geq 1 - \delta.
    \]
\end{proof}

For clarity, we restate Theorem~\ref{thm:additive-conditional-restart}.

\begin{theorem}
    Fix $L,R,\sigma>0$, an instance $\mathcal{P}\in\mathfrak{P}_{\mathrm{cvx}}(L,R,\sigma)$, and a deterministic nonnegative stepsize schedule $\eta$. Let $m\in\mathbb{N}$ and $r\in\mathbb{N}_+$ be deterministic, and suppose that
    \( \eta_{m} = \cdots = 
    \eta_{m+r-1} = \gamma \), \( 0 < \gamma \leq 1/(2L) \). Then for every almost surely finite
    \( \mathcal{F}_{m} \)-measurable comparator \( y \) and every \( \delta \in (0, 1) \),
    \[
        \mathbb{P} \bigg(
            f(x_{m+r}) - f(y) \leq \frac{4D(y, x_{m})}{\gamma r} + 6\sigma^2\gamma \mathsf{H}_{r}
            + 9\sigma^2\gamma \log \frac{1}{\delta} \bigg| \mathcal{F}_{m}
        \bigg) \geq 1 - \delta \quad \text{almost surely.}
    \]
\end{theorem}

\begin{proof}
    Condition on \(\mathcal{F}_m\), set \(X^t:=x_{m+t-1}\), \(Z^t:=Z_{m+t-1}\), and \(T:=r\),
    and use the local filtration \(\mathcal{G}_t:=\mathcal{F}_{m+t}\) for \(0\leq t\leq r\). Then
    \(X^1\) and \(y\) are \(\mathcal{G}_0\)-measurable, while \(Z^t\) is
    \(\mathcal{G}_t\)-measurable and satisfies the conditional noise assumptions given
    \(\mathcal{G}_{t-1}\). Lemma~\ref{lem:terminal-inequality} therefore applies conditionally and
    gives the stated bound.
\end{proof}

%% file: sections/Appendix-B-2-global-radius.tex
\subsection{Global radius and calibration} \label{appendix:global-radius}

For a time-uniform analysis based on finite stepsize energy, we must first control the entire trajectory.
The next lemma obtains this control from two nonnegative supermartingales: one for the accumulated noise
energy and one for the stopped directional noise.

Define \( D_t:=\|x_t-x^*\|^2/2 \), \( Q_\infty:=\sum_{t=0}^\infty\eta_t^2 \), and
\( \eta_*:=\sup_{t\geq0}\eta_t \). We assume \(D_0\leq R^2/2\), \(Q_\infty<\infty\), and
\(0\leq\eta_t\leq1/(2L)\).

\begin{lemma} \label{lem:global-radius}
    Let \(\beta\in(0,1)\) and \(\ell_\beta:=\log(2/\beta)\). Define
    \[
        C_\beta:=\frac{R^2}{2}+\frac{\sigma^2Q_\infty}{2}
        +\frac{\sigma^2\eta_*^2\ell_\beta}{2}, \quad
        q_\beta:=\sigma^2Q_\infty\ell_\beta,
    \]
    and set
    \[
        B_\beta:=\left(\sqrt{C_\beta+4q_\beta}+2\sqrt{q_\beta}\right)^2.
    \]
    Then
    \[
        \mathbb{P}\left(D_t\leq B_\beta,\ \forall t\geq0\right)\geq1-\beta.
    \]
\end{lemma}

\begin{proof}
    We divide the proof into four steps.

    \textit{Step 1: A pathwise distance recursion.}
    Set \(\bar{x}_{t+1}:=x_t-\eta_t\nabla f(x_t)\). Since \(f\) is convex and \(L\)-smooth and
    \(\nabla f(x^*)=0\), cocoercivity gives
    \[
        \langle\nabla f(x_t),x_t-x^*\rangle\geq\frac{1}{L}\|\nabla f(x_t)\|^2.
    \]
    Therefore, for every \(0\leq\eta_t\leq2/L\),
    \begin{align*}
        \|\bar{x}_{t+1}-x^*\|^2
        &=\|x_t-x^*\|^2-2\eta_t\langle\nabla f(x_t),x_t-x^*\rangle
        +\eta_t^2\|\nabla f(x_t)\|^2 \\
        &\leq\|x_t-x^*\|^2-\eta_t\left(\frac{2}{L}-\eta_t\right)
        \|\nabla f(x_t)\|^2 \\
        &\leq\|x_t-x^*\|^2.
    \end{align*}
    Thus the deterministic gradient step is nonexpansive relative to \(x^*\). Since
    \(x_{t+1}=\bar{x}_{t+1}-\eta_tZ_t\), we obtain the pathwise inequality
    \begin{equation}
        D_{t+1}\leq D_t+\frac{\eta_t^2}{2}\|Z_t\|^2
        -\eta_t\langle Z_t,\bar{x}_{t+1}-x^*\rangle.
        \label{eq:radius-one-step}
    \end{equation}
        
    \textit{Step 2: Time-uniform control of the accumulated noise energy.}
    If \(\eta_*=0\), then every step is zero and the result is immediate. Hence, assume \(\eta_*>0\)
    and put \(c_t:=\eta_t^2/\eta_*^2\in[0,1]\). The conditional noise assumption and conditional
    Jensen's inequality give
    \[
        \mathbb{E}\left[\left.
            \exp\left(c_t\left(\frac{\|Z_t\|^2}{\sigma^2}-1\right)\right)
        \right|\mathcal{F}_t\right]
        \leq\left\{\mathbb{E}\left[\left.
            \exp\left(\frac{\|Z_t\|^2}{\sigma^2}-1\right)
        \right|\mathcal{F}_t\right]\right\}^{c_t}
        \leq1.
    \]
    Consequently,
    \[
        \mathcal{S}_n:=\exp\left(
            \sum_{t=0}^{n-1}c_t\left(\frac{\|Z_t\|^2}{\sigma^2}-1\right)
        \right), \quad n\geq0,
    \]
    is a nonnegative supermartingale with \(\mathcal{S}_0=1\). Ville's inequality shows that, with
    probability at least \(1-\beta/2\), simultaneously for every \(n\geq0\),
    \begin{equation} \label{eq:radius-noise-energy}
        \sum_{t=0}^{n-1}\eta_t^2\|Z_t\|^2
        \leq\sigma^2\sum_{t=0}^{n-1}\eta_t^2+\sigma^2\eta_*^2\ell_\beta
        \leq\sigma^2Q_\infty+\sigma^2\eta_*^2\ell_\beta.
    \end{equation}

    \textit{Step 3: Time-uniform control of the stopped directional noise.}
    We first record the directional consequence of the conditional noise assumption. For every
    \(a,y\in\mathbb{R}\),
    \begin{align*}
        \mathrm{e}^{ay}-1-ay
        &\leq\frac{a^2y^2}{2}\mathrm{e}^{|ay|} \\
        &\leq\frac{a^2}{2}\mathrm{e}^{a^2/2}
        y^2\mathrm{e}^{y^2/2} \\
        &\leq\frac{a^2}{2}\mathrm{e}^{a^2/2}
        \left(\mathrm{e}^{y^2}-1\right).
    \end{align*}
    If \(u_t\neq0\), set \(Y_t:=\langle Z_t,u_t\rangle/(\sigma\|u_t\|)\). Conditional centering
    gives \(\mathbb{E}[Y_t\mid\mathcal{F}_t]=0\), while the conditional norm-square MGF gives
    \(\mathbb{E}[\mathrm{e}^{Y_t^2}\mid\mathcal{F}_t]\leq\mathrm{e}\). Therefore,
    \[
        \mathbb{E}\left[\left.\mathrm{e}^{aY_t}\right|\mathcal{F}_t\right]
        \leq1+\frac{\mathrm{e}-1}{2}a^2\mathrm{e}^{a^2/2}
        \leq\mathrm{e}^{2a^2}.
    \]
    Here the last inequality follows from \((\mathrm{e}-1)/2<1\) and
    \(s\mathrm{e}^{s/2}\leq\mathrm{e}^{2s}-1\) for \(s\geq0\). The case \(u_t=0\) is immediate.
    Taking \(a=\lambda\sigma\|u_t\|\) yields
    \begin{equation} \label{eq:conditional-directional-mgf}
        \mathbb{E}\left[\left.\exp\left(\lambda\langle Z_t,u_t\rangle\right)
        \right|\mathcal{F}_t\right]
        \leq\exp\left(2\sigma^2\lambda^2\|u_t\|^2\right)
    \end{equation}
    for every \(\mathcal{F}_t\)-measurable \(u_t\in\mathcal{H}\) and every \(\lambda\in\mathbb{R}\).
    Fix \(b>0\), and define the stopping time \(\tau_b:=\inf\{t\geq0:D_t>b\}\). On
    \(\{t<\tau_b\}\), Step 1 gives
    \[
        \|\bar{x}_{t+1}-x^*\|^2\leq\|x_t-x^*\|^2=2D_t\leq2b.
    \]
    Define the stopped sum
    \[
        M_n^{(b)}:=-\sum_{t=0}^{n-1}\eta_t
        \langle Z_t,\bar{x}_{t+1}-x^*\rangle\mathbf{1}_{\{t<\tau_b\}}.
    \]
    For any fixed \(\lambda>0\), define
    \[
        \mathcal{E}_n^{(b)}(\lambda):=\exp\left(
            \lambda M_n^{(b)}-4\sigma^2b\lambda^2\sum_{t=0}^{n-1}\eta_t^2
        \right).
    \]
    The indicator \(\mathbf{1}_{\{t<\tau_b\}}\) and the vector \(\bar{x}_{t+1}-x^*\) are
    \(\mathcal{F}_t\)-measurable. Hence, by \eqref{eq:conditional-directional-mgf},
    \begin{align*}
        \mathbb{E}\left[\left.\mathcal{E}_{n+1}^{(b)}(\lambda)\right|\mathcal{F}_n\right]
        &\leq\mathcal{E}_n^{(b)}(\lambda)
        \exp\left(2\sigma^2\lambda^2\eta_n^2
        \|\bar{x}_{n+1}-x^*\|^2\mathbf{1}_{\{n<\tau_b\}}
        -4\sigma^2b\lambda^2\eta_n^2\right) \\
        &\leq\mathcal{E}_n^{(b)}(\lambda).
    \end{align*}
    Thus \(\{\mathcal{E}_n^{(b)}(\lambda)\}_{n\geq0}\) is a nonnegative supermartingale starting
    from one. Ville's inequality gives, with probability at least \(1-\beta/2\), simultaneously for all
    \(n\geq0\),
    \[
        M_n^{(b)}\leq4\sigma^2b\lambda Q_\infty+\frac{\ell_\beta}{\lambda}.
    \]
    Taking \(\lambda=\sqrt{\ell_\beta/(4\sigma^2bQ_\infty)}\) yields
    \begin{equation}
        M_n^{(b)}\leq4\sigma\sqrt{bQ_\infty\ell_\beta}, \quad \forall n\geq0.
        \label{eq:radius-directional-noise}
    \end{equation}

    \textit{Step 4: Stopping-time contradiction and calibration.}
    Intersect the events in \eqref{eq:radius-noise-energy} and
    \eqref{eq:radius-directional-noise}; this intersection has probability at least \(1-\beta\).
    Set \(b=B_\beta\), and suppose on this event that \(\tau_b<\infty\). Summing
    \eqref{eq:radius-one-step} up to \(\tau_b\) gives
    \begin{align*}
        D_{\tau_b}
        &\leq D_0+\frac{1}{2}\sum_{t=0}^{\tau_b-1}\eta_t^2\|Z_t\|^2+M_{\tau_b}^{(b)} \\
        &\leq C_\beta+4\sqrt{bq_\beta}.
    \end{align*}
    The definition of \(B_\beta\) is exactly the positive solution of
    \(b=C_\beta+4\sqrt{bq_\beta}\). Therefore, \(D_{\tau_b}\leq b\), contradicting the definition
    of \(\tau_b\). Hence \(\tau_b=\infty\), and \(D_t\leq B_\beta\) for every \(t\geq0\) on the
    same event.
\end{proof}

\begin{remark}
    For the dyadic schedule detailed in Algorithm~\ref{alg:schedule}, we have
\[
    Q_{\infty} = \sum_{j=j_0}^{\infty} N_{j} \gamma_{j}^{2} = \sum_{j=j_0}^{\infty} N_{j} 
    \frac{R^2}{\sigma^2 c_h^2 N_j h_j^2} = \frac{R^2}{\sigma^2} c_h^{-2} \sum_{j=j_0}^{\infty} 
    \frac{1}{h_j^2} \leq \frac{R^2}{\sigma^2}.
\]
Then Lemma~\ref{lem:global-radius} holds. Suppose that we assign probability \( \alpha/2 \) to this
event. Then we have
\begin{align*}
    C_{\alpha/2} &= \frac{R^2}{2} + \frac{\sigma^2 Q_{\infty}}{2} 
    + \frac{\sigma^2 \eta_*^2 \ell_{\alpha/2}}{2} \\
    &\leq \frac{R^2}{2} + \frac{R^2}{2} + \frac{\sigma^2 \ell_{\alpha/2}}{2} \cdot \frac{R^2}{\sigma^2
    c_h^2 N_{j_0} h_{j_0}^2} \leq R^2 + \frac{R^2 \ell_{\alpha/2}}{2}.
\end{align*}
The last inequality follows from the fact that \( N_{j_0}, h_{j_0}, c_h \) are all at least \(1\).
Together with the fact that \( q_{\alpha/2} = \sigma^2 Q_{\infty} \ell_{\alpha/2} \leq R^2 \ell_{\alpha/2} \),
we have
\[
    B_{\alpha/2} = (\sqrt{C_{\alpha/2} + 4 q_{\alpha/2}} + 2 \sqrt{q_{\alpha/2}})^2
    \leq R^2 \left( \sqrt{ 1 + \frac{9}{2} \ell_{\alpha/2} } + 2 \sqrt{\ell_{\alpha/2}} \right)^2.
\]
Since \( \ell_{\alpha/2} = \log (4/\alpha) \geq 2\log 2 \) and
\[
    \sqrt{1 + \frac{9}{2} y} \leq \frac{5}{2} \sqrt{y}, \quad \forall y \geq 2\log 2,
\]
we have
\[
    B_{\alpha/2} \leq \frac{81}{4} R^2 \log \frac{4}{\alpha}.
\]
\end{remark}

%% file: sections/Appendix-B-3-sufficiency-proof.tex
\subsection{Proof of the sufficiency theorem} \label{appendix:sufficiency-proof}

Throughout this subsection, \( h \) is eventually nondecreasing with
\( \sum_{j \geq 1} h(2^j)^{-2} < \infty \), and \( \eta = \eta(h, L, R, \sigma) \) is the schedule
produced by Algorithm~\ref{alg:schedule}, with the associated quantities \( J_h \), \( c_h \),
\( j_0 \), \( N_j = 2^j \), \( h_j = h(N_j) \), and \( \gamma_j = R / (\sigma c_h \sqrt{N_j} h_j) \).
Recall \( \Delta_t = f(x_t) - f^* \) and \( D(y, x) = \| y - x \|^2 / 2 \).

We first record the two calibration properties of the schedule that make
Theorem~\ref{thm:additive-conditional-restart} and Lemma~\ref{lem:global-radius} applicable.

\begin{lemma}[Schedule calibration and finite energy] \label{lem:schedule-calibration}
    The schedule of Algorithm~\ref{alg:schedule} satisfies:
    \begin{enumerate}[label=(\roman*)]
        \item \( \{ \gamma_j \}_{j \geq j_0} \) is nonincreasing, and
        \( 0 < \gamma_j \leq 1 / (2L) \) for every \( j \geq j_0 \);
        \item the total squared-step energy is finite, with
        \[
            Q_{\infty} := \sum_{t \geq 0} \eta_t^2
            = \sum_{j = j_0}^{\infty} N_j \gamma_j^2
            \leq \frac{R^2}{\sigma^2}.
        \]
    \end{enumerate}
\end{lemma}

\begin{proof}
    Since \( j_0 \geq J_h \), the sequence \( \{ h_j \}_{j \geq j_0} \) is nondecreasing, and
    \( N_j \) is increasing. Hence \( \sqrt{N_j} h_j \) is increasing and
    \( \gamma_j = R / (\sigma c_h \sqrt{N_j} h_j) \) is nonincreasing. The definition of \( j_0 \)
    gives \( \gamma_{j_0} = R / (\sigma c_h \sqrt{N_{j_0}} h_{j_0}) \leq 1 / (2L) \), so
    \( \gamma_j \leq \gamma_{j_0} \leq 1/(2L) \) for every \( j \geq j_0 \). This proves (i).

    For (ii), the steps vanish before epoch \( j_0 \) and are constant on each later epoch, so
    \[
        Q_{\infty}
        = \sum_{j = j_0}^{\infty} N_j \gamma_j^2
        = \sum_{j = j_0}^{\infty} N_j \frac{R^2}{\sigma^2 c_h^2 N_j h_j^2}
        = \frac{R^2}{\sigma^2 c_h^2} \sum_{j = j_0}^{\infty} \frac{1}{h_j^2}
        \leq \frac{R^2}{\sigma^2},
    \]
    because \( j_0 \geq J_h \) and \( c_h^2 \geq \sum_{j \geq J_h} h_j^{-2} \) by construction.
\end{proof}

The next lemma is where reciprocal-square summability enters the constructive direction: it converts
Lemma~\ref{lem:summability-epoch} into two constants that do not depend on the epoch index.

\begin{lemma}[Epoch constants] \label{lem:epoch-constants}
    For every \( \alpha \in (0, 1) \), the quantities
    \[
        K_{h_1} := \sup_{j \geq j_0} \frac{(24 \log 2) j + 18 \log 2 + 6}{h_j^2},
        \qquad
        K_{h_2} := \sup_{j \geq j_0} \frac{9 \log (1 / \alpha)}{h_j^2}
    \]
    are finite. Moreover \( K_{h_2} = 9 \log(1/\alpha) / h_{j_0}^2 \leq 9 \log (1/\alpha) \).
\end{lemma}

\begin{proof}
    By the definition of \( j_0 \) we have \( h_{j_0} \geq 1 \), and \( \{ h_j \}_{j \geq j_0} \) is
    nondecreasing, so \( h_j \geq 1 \) for every \( j \geq j_0 \). This gives the stated evaluation of
    \( K_{h_2} \). For \( K_{h_1} \), Lemma~\ref{lem:summability-epoch} gives \( j / h_j^2 \to 0 \),
    while \( (18 \log 2 + 6) / h_j^2 \to 0 \). Hence the sequence whose supremum defines
    \( K_{h_1} \) converges to zero, and is therefore bounded.
\end{proof}

We now fix the confidence allocation. For \( j \geq j_0 \) and \( 1 \leq r \leq N_j \), set
\begin{equation} \label{eq:confidence-allocation}
    \delta_{j, r} := \frac{\alpha}{2^{2(j+1)}},
    \qquad
    y_{j, r} :=
    \begin{cases}
        x_{N_j}, & 1 \leq r < N_j / 2, \\
        x^*, & N_j / 2 \leq r \leq N_j,
    \end{cases}
\end{equation}
and let \( \mathcal{A}_{j, r} \) denote the event
\[
    \mathcal{A}_{j,r} :=
    \left\{
        f(x_{N_j + r}) - f(y_{j,r})
        \leq \frac{4 D(y_{j,r}, x_{N_j})}{\gamma_j r}
        + 6 \sigma^2 \gamma_j \mathsf{H}_r
        + 9 \sigma^2 \gamma_j \log \frac{1}{\delta_{j,r}}
    \right\}.
\]
Both choices of \( y_{j,r} \) are almost surely finite and \( \mathcal{F}_{N_j} \)-measurable, and
\( \gamma_j \leq 1/(2L) \) by Lemma~\ref{lem:schedule-calibration}. Theorem~\ref{thm:additive-conditional-restart}
therefore applies with \( m = N_j \) and this \( r \), giving
\( \mathbb{P}(\mathcal{A}_{j,r} \mid \mathcal{F}_{N_j}) \geq 1 - \delta_{j,r} \) and hence
\( \mathbb{P}(\mathcal{A}_{j,r}) \geq 1 - \delta_{j,r} \). Let
\begin{equation} \label{eq:global-radius-event}
    \mathcal{G} := \left\{ D(x^*, x_t) \leq B_{\alpha/2}, \; \forall t \geq 0 \right\},
    \qquad
    B_{\alpha/2} \leq \frac{81}{4} R^2 \log \frac{4}{\alpha},
\end{equation}
which by Lemma~\ref{lem:schedule-calibration} and Lemma~\ref{lem:global-radius} satisfies
\( \mathbb{P}(\mathcal{G}) \geq 1 - \alpha / 2 \).

\begin{lemma}[Within-epoch estimate] \label{lem:within-epoch}
    Let \( j \geq j_0 \) and \( 1 \leq r \leq N_j \). On the event \( \mathcal{G} \cap \mathcal{A}_{j,r} \):
    \begin{enumerate}[label=(\roman*)]
        \item if \( 1 \leq r < N_j / 2 \), then
        \[
            f(x_{N_j + r}) - f(x_{N_j})
            \leq \frac{K_{h_1} + K_{h_2}}{c_h} \cdot R \sigma \frac{h_j}{\sqrt{N_j}};
        \]
        \item if \( N_j / 2 \leq r \leq N_j \), then
        \[
            \Delta_{N_j + r}
            \leq \left( 162 c_h \log \frac{4}{\alpha} + \frac{K_{h_1} + K_{h_2}}{c_h} \right)
            R \sigma \frac{h_j}{\sqrt{N_j}}.
        \]
    \end{enumerate}
\end{lemma}

\begin{proof}
    We first bound the stochastic terms, which are common to both cases. Since \( r \leq N_j \), we have
    \( \mathsf{H}_r \leq \log r + 1 \leq j \log 2 + 1 \), and
    \( \log (1 / \delta_{j,r}) = 2(j+1) \log 2 + \log (1/\alpha) \) by
    \eqref{eq:confidence-allocation}. Using
    \( \sigma^2 \gamma_j = R \sigma / (c_h \sqrt{N_j} h_j) \), we obtain
    \begin{align*}
        6 \sigma^2 \gamma_j \mathsf{H}_r + 9 \sigma^2 \gamma_j \log \frac{1}{\delta_{j,r}}
        &\leq \frac{R \sigma}{c_h \sqrt{N_j} h_j}
        \left( 6 (j \log 2 + 1) + 9 \left( 2(j+1) \log 2 + \log \frac{1}{\alpha} \right) \right) \\
        &= \frac{1}{c_h} \cdot R \sigma \frac{h_j}{\sqrt{N_j}}
        \left( \frac{(24 \log 2) j + 18 \log 2 + 6}{h_j^2}
        + \frac{9 \log (1/\alpha)}{h_j^2} \right) \\
        &\leq \frac{K_{h_1} + K_{h_2}}{c_h} \cdot R \sigma \frac{h_j}{\sqrt{N_j}},
    \end{align*}
    the last step by Lemma~\ref{lem:epoch-constants}.

    For (i), the comparator is \( y_{j,r} = x_{N_j} \), so the initial distance term vanishes:
    \( D(x_{N_j}, x_{N_j}) = 0 \). On \( \mathcal{A}_{j,r} \) the displayed bound is exactly the
    stochastic estimate above.

    For (ii), the comparator is \( y_{j,r} = x^* \), so
    \( f(x_{N_j + r}) - f(y_{j,r}) = \Delta_{N_j + r} \). On \( \mathcal{G} \) we have
    \( D(x^*, x_{N_j}) \leq B_{\alpha/2} \), and \( r \geq N_j / 2 \), so by
    \eqref{eq:global-radius-event},
    \[
        \frac{4 D(x^*, x_{N_j})}{\gamma_j r}
        \leq \frac{162 R^2 \log (4/\alpha)}{N_j \gamma_j}
        = \frac{162 R^2 \log (4/\alpha)}{N_j} \cdot \frac{\sigma c_h \sqrt{N_j} h_j}{R}
        = 162 c_h \log \frac{4}{\alpha} \cdot R \sigma \frac{h_j}{\sqrt{N_j}}.
    \]
    Adding the stochastic estimate gives (ii).
\end{proof}

The latter-half estimate at the end of epoch \( j - 1 \) is exactly the quantity needed to initialize
the early half of epoch \( j \). This is what propagates a single constant across all active epochs.

\begin{lemma}[Epoch propagation] \label{lem:epoch-propagation}
    Let
    \[
        C_{\mathrm{latter}} := \left( 324 c_h \log \frac{4}{\alpha}
        + \frac{(2 + \sqrt{2})(K_{h_1} + K_{h_2})}{c_h} \right) R \sigma.
    \]
    On the event \( \mathcal{G} \cap \bigcap_{j \geq j_0} \bigcap_{1 \leq r \leq N_j} \mathcal{A}_{j,r} \),
    we have
    \[
        \Delta_t \leq C_{\mathrm{latter}} \frac{h(t)}{\sqrt{t}},
        \qquad \forall t \geq \frac{3 N_{j_0}}{2}.
    \]
\end{lemma}

\begin{proof}
    Fix \( j > j_0 \) and \( 1 \leq r < N_j / 2 \). Since \( N_j = N_{j-1} + N_{j-1} \), the index
    \( N_j \) is the terminal position of epoch \( j - 1 \) reached with \( r = N_{j-1} \), so
    Lemma~\ref{lem:within-epoch}(ii) applied to epoch \( j - 1 \) gives
    \[
        \Delta_{N_j}
        \leq \left( 162 c_h \log \frac{4}{\alpha} + \frac{K_{h_1} + K_{h_2}}{c_h} \right)
        R \sigma \frac{h_{j-1}}{\sqrt{N_{j-1}}}.
    \]
    Adding this to the increment bound of Lemma~\ref{lem:within-epoch}(i) yields
    \begin{align*}
        \Delta_{N_j + r}
        &= \bigl( f(x_{N_j + r}) - f(x_{N_j}) \bigr) + \Delta_{N_j} \\
        &\leq \frac{K_{h_1} + K_{h_2}}{c_h} R \sigma \frac{h_j}{\sqrt{N_j}}
        + \left( 162 c_h \log \frac{4}{\alpha} + \frac{K_{h_1} + K_{h_2}}{c_h} \right)
        R \sigma \frac{h_{j-1}}{\sqrt{N_{j-1}}}.
    \end{align*}
    Since \( h_{j-1} \leq h_j \) and \( N_j = 2 N_{j-1} \), we have
    \( h_{j-1} / \sqrt{N_{j-1}} \leq \sqrt{2} h_j / \sqrt{N_j} \), whence
    \begin{equation} \label{eq:epoch-propagated-bound}
        \Delta_{N_j + r}
        \leq \left( 162 \sqrt{2} c_h \log \frac{4}{\alpha}
        + \frac{(\sqrt{2} + 1)(K_{h_1} + K_{h_2})}{c_h} \right) R \sigma \frac{h_j}{\sqrt{N_j}}.
    \end{equation}

    Now let \( t \geq 3 N_{j_0} / 2 \) and write \( t = N_j + r \) with \( j \geq j_0 \) and
    \( 1 \leq r \leq N_j \). If \( r \geq N_j / 2 \), then \( \Delta_t \) obeys
    Lemma~\ref{lem:within-epoch}(ii); if \( r < N_j / 2 \), then necessarily \( j > j_0 \) and
    \( \Delta_t \) obeys \eqref{eq:epoch-propagated-bound}, which dominates the former bound. In both
    cases
    \[
        \Delta_t \leq \left( 162 \sqrt{2} c_h \log \frac{4}{\alpha}
        + \frac{(\sqrt{2} + 1)(K_{h_1} + K_{h_2})}{c_h} \right) R \sigma \frac{h_j}{\sqrt{N_j}}.
    \]
    Finally, \( N_j \leq t \leq 2 N_j \) and \( h_j \leq h(t) \) by monotonicity, so
    \( h_j / \sqrt{N_j} \leq \sqrt{2} h(t) / \sqrt{t} \). Multiplying the bracket by \( \sqrt{2} \)
    gives \( C_{\mathrm{latter}} \).
\end{proof}

It remains to absorb the finitely many iterates preceding \( 3 N_{j_0} / 2 \).

\begin{lemma}[Prefix absorption] \label{lem:sufficiency-prefix}
    Let
    \[
        C_{\mathrm{prefix}} := \left( \frac{L R^2}{2}
        + \frac{K_{h_1} + K_{h_2}}{c_h} \cdot R \sigma \frac{h_{j_0}}{\sqrt{N_{j_0}}} \right)
        \bigg/ \min_{1 \leq n < 3 N_{j_0} / 2} \min \left\{ 1, \frac{h(n)}{\sqrt{n}} \right\}.
    \]
    On the event \( \mathcal{G} \cap \bigcap_{1 \leq r < N_{j_0} / 2} \mathcal{A}_{j_0, r} \), we have
    \[
        \Delta_t \leq C_{\mathrm{prefix}} \min \left\{ 1, \frac{h(t)}{\sqrt{t}} \right\},
        \qquad \forall\, 1 \leq t < \frac{3 N_{j_0}}{2}.
    \]
\end{lemma}

\begin{proof}
    Algorithm~\ref{alg:schedule} sets \( \eta_t = 0 \) for every \( t < N_{j_0} \), so the iterates
    remain at the initial point and \( x_{N_{j_0}} = x_0 \). Smoothness at the minimizer gives the
    deterministic bound
    \begin{equation} \label{eq:first-active-epoch-base}
        \Delta_{N_{j_0}} = \Delta_t = \Delta_0
        \leq \frac{L}{2} \| x_0 - x^* \|^2 \leq \frac{L R^2}{2},
        \qquad \forall\, 1 \leq t \leq N_{j_0}.
    \end{equation}
    For \( 1 \leq r < N_{j_0} / 2 \), combining \eqref{eq:first-active-epoch-base} with
    Lemma~\ref{lem:within-epoch}(i) applied at \( j = j_0 \) gives
    \[
        \Delta_{N_{j_0} + r}
        = \bigl( f(x_{N_{j_0} + r}) - f(x_{N_{j_0}}) \bigr) + \Delta_{N_{j_0}}
        \leq \frac{L R^2}{2}
        + \frac{K_{h_1} + K_{h_2}}{c_h} \cdot R \sigma \frac{h_{j_0}}{\sqrt{N_{j_0}}}.
    \]
    The same constant dominates \eqref{eq:first-active-epoch-base}. Since the range
    \( 1 \leq t < 3 N_{j_0} / 2 \) is finite and \( \min \{ 1, h(n) / \sqrt{n} \} > 0 \) on it,
    dividing by that minimum gives the claim.
\end{proof}

%% file: sections/Appendix-C-auxiliary-results-necessity.tex
\section{Proofs and auxiliary results for necessity}
\label{appendix:auxiliary-results-necessity}

\subsection{Fixed-time analysis}
\label{appendixsub:fixed-time-analysis}

This subsection proves Lemmas~\ref{lem:deterministic-progress-obstruction} and
\ref{lem:stochastic-energy-obstruction}. For \( a > 0 \) and \( 0 < \lambda \leq L \), define
\[
    \varphi_{\lambda, a}(y) := \frac{\lambda}{2}
    \left(\sqrt{a^2 + y^2} - a\right)^2, \quad
    q_{\lambda}(y) := \frac{\lambda}{2} y^2.
\]

\begin{lemma}[Geometry of the flat-active family]
    \label{lem:appendix-flat-active-geometry}
   The function \(\varphi_{\lambda,a}\) is real analytic and convex on \(\mathbb{R}\). Moreover, it is \(\lambda\)-smooth, and hence \(L\)-smooth. Its unique minimizer is \(y^*=0\); however, it is not strongly convex on any neighborhood of \(y^*\) and does not satisfy a PL inequality with any positive constant. Furthermore, for every \(y\in\mathbb{R}\),
\begin{align*}
\left|\varphi_{\lambda,a}'(y)-\lambda y\right|
&\leq \lambda a,\quad
0\leq q_\lambda(y)-\varphi_{\lambda,a}(y)
\leq \lambda a|y|.
\end{align*}
\end{lemma}

\begin{proof}
    Let \( s := \sqrt{a^2 + y^2} \). Direct differentiation gives
    \[
        \varphi_{\lambda, a}'(y) = \lambda y \left(1 - \frac{a}{s}\right), \quad
        \varphi_{\lambda, a}''(y) = \lambda
        \left(1 - \frac{a^3}{(a^2 + y^2)^{3/2}}\right) \in [0, \lambda].
    \]
    Thus, \( \varphi_{\lambda, a} \) is convex and \( \lambda \)-smooth, and hence also
    \( L \)-smooth. It is nonnegative and vanishes only at zero, so zero is its unique minimizer.
    Since \( \varphi_{\lambda, a}''(0) = 0 \), it is not strongly convex on any neighborhood of
    zero. For \( y \neq 0 \),
    \[
        \frac{\varphi_{\lambda, a}'(y)^2}{2\varphi_{\lambda, a}(y)}
        = \lambda \frac{y^2}{a^2 + y^2} \to 0, \quad \text{as } y \to 0,
    \]
    which rules out a PL inequality with a positive constant. The identity
    \[
        \sqrt{a^2 + y^2} - a = \frac{y^2}{\sqrt{a^2 + y^2} + a}
    \]
    further gives
    \[
        \varphi_{\lambda, a}(y) = \frac{\lambda}{8a^2} y^4 + \mathcal{O}(y^6),
        \quad \text{as } y \to 0.
    \]
    Therefore, the objective is quartic-flat at its minimizer.

    Finally,
    \[
        |\varphi_{\lambda, a}'(y) - \lambda y|
        = \frac{\lambda a |y|}{\sqrt{a^2 + y^2}}
        \leq \lambda a.
    \]
    Using \( y^2 = (s-a)(s+a) \), we also obtain
    \[
        q_{\lambda}(y) - \varphi_{\lambda, a}(y)
        = \lambda a(s-a) \in [0, \lambda a|y|].
    \]
\end{proof}

We next show that, over a fixed number of iterations, the recursion for
\( \varphi_{\lambda,a} \) converges pathwise to the corresponding quadratic recursion as
\( a \downarrow 0 \).

\begin{lemma}[Finite-horizon flat transfer]
    \label{lem:appendix-quartic-to-quadratic}
    Fix \( n \in \mathbb{N}_{+} \) and suppose \( \lambda S_n \leq 1/2 \). From the same initial
    point and with the same noise sequence, define
    \begin{align*}
        Y_{t+1}^{(a)} &= Y_t^{(a)} - \eta_t
        \left(\varphi_{\lambda,a}'(Y_t^{(a)}) + \zeta_t\right), \\
        Y_{t+1}^{(0)} &= Y_t^{(0)} - \eta_t
        \left(\lambda Y_t^{(0)} + \zeta_t\right).
    \end{align*}
    Then, pathwise,
    \[
        \max_{0 \leq t \leq n}|Y_t^{(a)} - Y_t^{(0)}|
        \leq \lambda aS_n \leq \frac{a}{2}.
    \]
    In particular,
    \[
        \varphi_{\lambda,a}(Y_n^{(a)}) \to q_{\lambda}(Y_n^{(0)}),
        \quad \text{as } a \downarrow 0.
    \]
\end{lemma}

\begin{proof}
    Let \( D_t := Y_t^{(a)} - Y_t^{(0)} \) and
    \( r_a(y) := \varphi_{\lambda,a}'(y) - \lambda y \). The coupled recursions imply
    \[
        D_{t+1} = (1 - \lambda\eta_t)D_t - \eta_t r_a(Y_t^{(a)}).
    \]
    Since \( 0 \leq \lambda\eta_t \leq \lambda S_n \leq 1/2 \),
    Lemma~\ref{lem:appendix-flat-active-geometry} gives
    \[
        |D_{t+1}| \leq (1 - \lambda\eta_t)|D_t| + \lambda a\eta_t.
    \]
    Starting from \( D_0 = 0 \), for every \( 1 \leq k \leq n \), we have
    \[
        |D_k| \leq \lambda a \sum_{i=0}^{k-1}\eta_i
        \prod_{r=i+1}^{k-1}(1 - \lambda\eta_r)
        \leq \lambda aS_n \leq \frac{a}{2}.
    \]
    This proves the uniform pathwise bound. In particular, \( Y_n^{(a)} \to Y_n^{(0)} \) as
    \( a \downarrow 0 \). Lemma~\ref{lem:appendix-flat-active-geometry} and the continuity of
    \( q_{\lambda} \) yield
    \begin{align*}
        |\varphi_{\lambda,a}(Y_n^{(a)}) - q_{\lambda}(Y_n^{(0)})|
        &\leq |\varphi_{\lambda,a}(Y_n^{(a)}) - q_{\lambda}(Y_n^{(a)})| 
        + |q_{\lambda}(Y_n^{(a)}) - q_{\lambda}(Y_n^{(0)})| \\
        &\leq \lambda a|Y_n^{(a)}| + \frac{\lambda}{2}
        |Y_n^{(a)} - Y_n^{(0)}|(|Y_n^{(a)}| + |Y_n^{(0)}|)
        \to 0.
    \end{align*}
\end{proof}

The preceding pathwise convergence also transfers any marginal coverage inequality from the
quartic-flat family to the limiting quadratic model.

\begin{lemma}[Transfer of marginal coverage]
    \label{lem:appendix-marginal-flat-transfer}
    Fix \( n \in \mathbb{N}_{+} \), an initial point, and a noise distribution, and suppose
    \( \lambda S_n \leq 1/2 \). If
    \[
        \mathbb{P}\left(\varphi_{\lambda,a}(Y_n^{(a)})
        \leq \varepsilon_{\alpha}(n)\right) \geq 1 - \alpha,
        \quad \forall a > 0,
    \]
    then
    \[
        \mathbb{P}\left(q_{\lambda}(Y_n^{(0)})
        \leq \varepsilon_{\alpha}(n)\right) \geq 1 - \alpha.
    \]
\end{lemma}

\begin{proof}
    Lemma~\ref{lem:appendix-quartic-to-quadratic} gives almost-sure convergence of the terminal
    objective values. Since \( (-\infty, \varepsilon_{\alpha}(n)] \) is closed, the closed-set part
    of the Portmanteau theorem gives
    \[
        \mathbb{P}\left(q_{\lambda}(Y_n^{(0)}) \leq \varepsilon_{\alpha}(n)\right)
        \geq \limsup_{a \downarrow 0}
        \mathbb{P}\left(\varphi_{\lambda,a}(Y_n^{(a)})
        \leq \varepsilon_{\alpha}(n)\right)
        \geq 1 - \alpha.
    \]
\end{proof}

We are now ready to prove the two fixed-horizon restrictions stated in the main text.

\begin{proof}[Proof of Lemma~\ref{lem:deterministic-progress-obstruction}]
    If \( S_n = 0 \), set \( \lambda_n := L/2 \). Otherwise, set
    \( \lambda_n := \min\{L, 1/S_n\}/2 \). In both cases, \( 0 < \lambda_n \leq L \) and
    \( \lambda_nS_n \leq 1/2 \). Start the quartic-flat recursion at \( Y_0^{(a)} = R \) and use
    zero noise. The limiting quadratic recursion satisfies
    \[
        Y_n^{(0)} = R\prod_{t<n}(1 - \lambda_n\eta_t).
    \]
    Every factor lies in \( [0,1] \). Therefore, the product inequality
    \( \prod_i(1-u_i) \geq 1 - \sum_i u_i \) gives
    \[
        |Y_n^{(0)}| \geq R(1 - \lambda_nS_n) \geq \frac{R}{2}.
    \]
    The zero-noise oracle is admissible for every \( a > 0 \). Time-uniform validity thus implies
    the marginal premise of Lemma~\ref{lem:appendix-marginal-flat-transfer}, and hence
    \[
        \varepsilon_{\alpha}(n)
        \geq q_{\lambda_n}(Y_n^{(0)})
        \geq \frac{\lambda_nR^2}{8}
        = \frac{R^2}{16}\min\left\{L, \frac{1}{S_n}\right\}.
    \]
\end{proof}

\begin{proof}[Proof of Lemma~\ref{lem:stochastic-energy-obstruction}]
    Set \( \lambda_n := 1/(2S_n) \), start from \( Y_0^{(a)} = 0 \), and take independent noises
    \( \zeta_t \sim \mathcal{N}(0, \sigma^2/4) \). This oracle is admissible because
    \[
        \mathbb{E}\left[\exp\left(\frac{\zeta_t^2}{\sigma^2}\right)\right]
        = \sqrt{2} < \mathrm{e}.
    \]
    Since \( S_n \geq 1/L \), we have \( \lambda_n \leq L/2 \), while
    \( \lambda_nS_n = 1/2 \). Lemma~\ref{lem:appendix-marginal-flat-transfer} therefore applies.
    The limiting quadratic recursion can be unrolled as
    \[
        Y_n^{(0)} = -\sum_{t<n}\eta_t\zeta_t
        \prod_{s=t+1}^{n-1}(1 - \lambda_n\eta_s).
    \]
    For every \( t<n \), the product inequality used above gives
    \[
        \prod_{s=t+1}^{n-1}(1 - \lambda_n\eta_s)
        \geq 1 - \lambda_n\sum_{s=t+1}^{n-1}\eta_s
        \geq \frac{1}{2}.
    \]
    Thus, \( Y_n^{(0)} \) is centered Gaussian with
    \[
        \operatorname{Var}(Y_n^{(0)})
        = \frac{\sigma^2}{4}\sum_{t<n}\eta_t^2
        \prod_{s=t+1}^{n-1}(1 - \lambda_n\eta_s)^2
        \geq \frac{\sigma^2Q_n}{16}.
    \]
    Marginal coverage of \( q_{\lambda_n}(Y_n^{(0)}) \) is equivalent to
    \[
        \mathbb{P}\bigg(
            |Y_n^{(0)}| \leq \sqrt{\frac{2\varepsilon_{\alpha}(n)}{\lambda_n}}
        \bigg) \geq 1 - \alpha.
    \]
    By the definition \( z_{\alpha} = \Phi^{-1}(1 - \alpha/2) \), the two-sided Gaussian quantile
    requires
    \[
        \frac{2\varepsilon_{\alpha}(n)}{\lambda_n}
        \geq z_{\alpha}^2\operatorname{Var}(Y_n^{(0)})
        \geq \frac{z_{\alpha}^2\sigma^2Q_n}{16}.
    \]
    Substituting \( \lambda_n = 1/(2S_n) \) proves the result.
\end{proof}

Corollary~\ref{cor:step-lower-bound} is a direct result from Lemma~\ref{lem:deterministic-progress-obstruction} and \ref{lem:stochastic-energy-obstruction}.

\begin{proof} [Proof of Corollary~\ref{cor:step-lower-bound}]
    Since \( h(n) = o(\sqrt{n}) \), the boundary \( \varepsilon_{\alpha}(n) \to 0 \) as \( n \to \infty \).
    Therefore, by Lemma~\ref{lem:deterministic-progress-obstruction}, for sufficiently large \( n \), we have \( S_{n} \geq 1/L \), and additionally, 
    \[
        \frac{R^2}{16} \frac{1}{S_{n}} \leq B \frac{h(n)}{\sqrt{n}} \implies 
        S_{n} \geq \frac{R^2}{16 B} \frac{\sqrt{n}}{h(n)} \to \infty.
    \]
    Thus, for sufficiently large \( j \), Lemma~\ref{lem:stochastic-energy-obstruction} applies, yielding
    \[
        \frac{z_{\alpha}^2 \sigma^2}{64} \frac{\mathcal{Q}_j}{S_{N_j}} \leq B \frac{h_j}{\sqrt{N_j}}
        \implies p_j \geq \frac{z_{\alpha}^2 \sigma^2}{64 B} \frac{\mathcal{Q}_j}{h_j}.
    \]
\end{proof}

\subsection{Preliminaries for the fixed Gaussian instance}
\label{appendixsub:preliminaries-Gaussian}

As mentioned in Section~\ref{subsec:necessity-witness}, the instance is given by
\[
    \phi(y) := \frac{L}{4} (\sqrt{R^2 + y^2} - R)^2 = \varphi_{L/2, R}(y).
\]
The oracle noise is Gaussian with variance \( \nu^2 := \vartheta \sigma^2 \), where
\( \vartheta := (1 - \mathrm{e}^{-2})/2 \). The recursion is initialized at the minimizer
and uses independent \( \zeta_t \sim \mathcal{N}(0, \nu^2) \) as the oracle noise:
\[
    Y_{t+1} = Y_{t} - \eta_t (\phi'(Y_t) + \zeta_t), \quad Y_0 = 0.
\]
We now give some properties of this instance that will be used in the necessity proof.

\begin{lemma} \label{lem:geometry-phi}
    The function \( \phi \) is convex, \( L \)-smooth, and non-PL. It is quartic-flat at the minimizer,
    and its derivative satisfies
    \[
        |\phi'(y)| \leq \frac{L|y|^3}{4R^2}, \quad \forall y \in \mathbb{R}.
    \]
\end{lemma}

\begin{proof}
    The first three properties follow from Lemma~\ref{lem:appendix-flat-active-geometry}. The derivative
    is given by
    \[
        \phi'(y) = \frac{Ly^3}{2\sqrt{R^2+y^2}(\sqrt{R^2+y^2}+R)}.
    \]
    Since \( \sqrt{R^2+y^2}(\sqrt{R^2+y^2}+R) \geq 2R^2 \), we have
    \[
        |\phi'(y)| \leq \frac{L|y|^3}{4R^2}.
    \]
\end{proof}

This instance forces the stepsize to vanish asymptotically, as stated in the following lemma.

\begin{lemma} \label{lem:deminishing-stepsize}
    Under the assumptions of Theorem~\ref{thm:necessity}, the stepsize sequence \( \{\eta_t\} \) must
    satisfy \( \eta_t \to 0 \) as \( t \to \infty \).
\end{lemma}

\begin{proof}
    For every \( u \geq 0 \), \( \phi(y) \leq u \) is equivalent to
    \[
        |y|^2 \leq r^2(u) := 4R \sqrt{\frac{u}{L}} + \frac{4u}{L}.
    \]
    It is easy to see that \( r^2(u) \to 0 \) as \( u \to 0^+ \). Since \( h(n) / \sqrt{n} \to 0 \),
    we have \( \varepsilon_{\alpha}(n) \to 0 \) as \( n \to \infty \). Hence, \(r^2(\varepsilon_{\alpha}(n))
    \to 0\) as \( n \to \infty \). Conditional on the past, \( Y_{t+1} \) is a translate of
    \( \mathcal{N}(0, \nu^2 \eta_t^2) \). When \( \eta_t > 0 \), its conditional density is bounded by
    \( (\sqrt{2\pi}\nu\eta_t)^{-1} \), regardless of the conditional mean. Integrating this bound over
    \( [-r(\varepsilon_{\alpha}(t+1)), r(\varepsilon_{\alpha}(t+1))] \) yields
    \[
        1-\alpha \leq \mathbb{P} \left( |Y_{t+1}| \leq r(\varepsilon_{\alpha}(t+1)) \right)
        \leq \sqrt{\frac{2}{\pi}} \frac{r(\varepsilon_{\alpha}(t+1))}{\nu\eta_t}.
    \]
    Thus every positive \( \eta_t \) is bounded by a constant multiple of a quantity tending to zero.
    Zero steps already satisfy the same conclusion, so \( \eta_t \to 0 \).
\end{proof}

In Section~\ref{subsec:necessity-witness}, we define \( N_j = 2^j \), \( h_j = h(N_j) \), and
\[
    \hat{h}_j := \max_{\substack{N_j \leq n \leq N_{j+1} \\ n \in \mathbb{N}_{+}}} h(n) \sqrt{\frac{N_j}{n}}.
\]
We also define \( \bar{\varepsilon}_j = B \hat{h}_j / \sqrt{N_j} \) and \(\beta_j^2 := r^2(\bar
{\varepsilon}_j)\). Now we will give a lower bound and an upper bound for \( \beta_j^4 \).

\begin{lemma} \label{lem:tolerance-bound}
    For sufficiently large \( j \in \mathbb{N}_+ \), we have
    \[
        \frac{16 R^2 \bar{\varepsilon}_j}{L} \leq \beta_j^4 \leq \frac{64 R^2 \bar{\varepsilon}_j}{L}.
    \]
\end{lemma}

\begin{proof}
    Since \( r^2(\varepsilon_j) \geq 4 R \sqrt{\varepsilon_j / L} \), the lower bound is immediate.
    For the upper bound, we notice that
    \[
        \bar{\varepsilon}_j = B \frac{\hat{h}_j}{\sqrt{N_j}} \leq B \frac{h_{j+1}}{\sqrt{N_j}}= \sqrt{2} B \frac{h_{j+1}}{\sqrt{N_{j+1}}}
        % = \sqrt{2} \varepsilon_{\alpha}(N_{j+1}) 
        \to 0, \quad \text{as } j \to \infty.
    \]
    Therefore, for sufficiently large \( j \), we have \( \bar{\varepsilon}_j \leq LR^2 \) and
    \[
        r^2(\bar{\varepsilon}_j) = 4R\sqrt{\frac{\bar{\varepsilon}_j}{L}} + \frac{4\bar{\varepsilon}_j}{L}
        \leq 8R \sqrt{\frac{\bar{\varepsilon}_j}{L}}.
    \]
    Hence, the desired upper bound holds for sufficiently large \( j \).
\end{proof}

\subsection{Greedy step-mass blocks}
\label{appendixsub:greedy-step-mass}

Now we define the accumulated stepsize over the \( j \)-th dyadic epoch as
\[
    H_j := \sum_{t=N_j}^{N_{j+1} - 1} \eta_t.
\]
We consider dividing the \( j \)-th epoch into several sub-blocks: let \( \tau_j > 0 \) be a tolerance
parameter, and define \( e_k \) as the smallest integer such that
\[
    H_j^{(k)} := \sum_{t=e_{k-1}}^{e_k - 1} \eta_t \geq \tau_j, \quad e_0 := N_j.
\]
We stop when the remaining step mass is less than \( \tau_j \). Let \( K \) be the number of complete
sub-blocks and \( m_k := e_k - e_{k-1} \) be the length of the \( k \)-th sub-block. Then we have
the following lemma.

\begin{lemma} \label{lem:sub-block-algebra}
    For sufficiently large \( j \in \mathbb{N}_{+} \), if \(\max_{N_j \leq t < N_{j+1}} \eta_t \leq \tau_j/4\), we have
    \[
        \tau_j \leq H_{j}^{(k)} \leq \frac{5}{4} \tau_j, \quad \sum_{k=1}^{K} m_k \leq N_j.
    \]
    Note that \( H_j = \sum_{t=N_j}^{N_{j+1} - 1} \eta_t \). We also have
    \[
        K \geq \frac{4}{5} \cdot \frac{H_j}{\tau_j} - \frac{4}{5}.
    \]
    In particular, if \( H_j \geq 2\tau_j \), then \( K \geq (2 H_j) / (5 \tau_j) \).
\end{lemma}

\begin{proof}
    \( \sum_{k=1}^{K} m_k \leq N_j \) is trivial since the sub-blocks are disjoint and contained in
    the \( j \)-th epoch. By the definition of \( e_k \), \( H_j^{(k)} \geq \tau_j \) for each \( k \).
    Since \( \eta_t \to 0 \) as \( t \to \infty \), we have \( \eta_t \leq \tau_j / 4 \) for sufficiently
    large \( j \) and \( t \in [N_j, N_{j+1}) \). Therefore, the minimality of \( e_k \) implies that
    the stepsize mass before the last step in the \( k \)-th sub-block is less than \( \tau_j \). Meanwhile,
    the last step contributes at most \( \tau_j / 4 \). Hence, \( H_j^{(k)} \leq 5\tau_j / 4 \).
    Finally, since
    \[
        \sum_{k=1}^{K} H_j^{(k)} \leq H_j \leq \sum_{k=1}^{K} H_j^{(k)} + \tau_j,
    \]
    we have
    \[
        H_j \leq \frac{5}{4} K \tau_j + \tau_j \implies K \geq \frac{4}{5} \cdot \frac{H_j}{\tau_j} - 
        \frac{4}{5}.
    \]
    If \( H_j \geq 2\tau_j \), then \( K \geq (2 H_j) / (5 \tau_j) \).
\end{proof}

Now we set \( \kappa_0 = 32/5 \), and \( \tau_j := \kappa_0 R^2 / (L \beta_j^2) \). 
Since \( \beta_j \to 0 \), we have \( \tau_j \to \infty \). Together with 
Lemma~\ref{lem:deminishing-stepsize}, this guarantees
\( \max_{N_j \leq t<N_{j+1}}\eta_t \leq \tau_j/4 \) for all sufficiently large \( j \).

\begin{lemma} \label{lem:block-noise-constraint}
    For a sufficiently large \( j \in \mathbb{N}_{+} \) with \( H_j / \tau_j \geq 2 \), we define
    \( G_k := \sum_{t=e_{k-1}}^{e_k - 1} \eta_t \zeta_t \), \( V_k := \operatorname{Var}(G_k) \).
    Then the variables \( G_1, \dots, G_K \) are independent
    centered Gaussian, and
    \[
        \mathcal{C} \subset \bigcap_{k=1}^{K} \{ |G_k| \leq 4 \beta_j \}.
    \]
    Moreover, we have
    \[
        \frac{1}{K} \sum_{k=1}^{K} \frac{\beta_j^2}{V_k} \leq \frac{25}{\vartheta} B \frac{\hat{h}_j
        \sqrt{N_j}}{\sigma^2 H_j}.
    \]
\end{lemma}

\begin{proof}
    Independence and zero mean are trivial since the \( G_k \)'s are linear combinations of disjoint sets of
    independent centered Gaussian variables. Notice that
    \[
        Y_{e_k} - Y_{e_{k-1}} = -\sum_{t=e_{k-1}}^{e_k - 1} \eta_t (\phi'(Y_t) + \zeta_t)
        = -\sum_{t=e_{k-1}}^{e_k - 1} \eta_t \phi'(Y_t) - G_k.
    \]
    Now we control the deterministic term. By Lemma~\ref{lem:geometry-phi}, on the event \( \mathcal{C} \),
    we have
    \[
        \bigg| \sum_{t=e_{k-1}}^{e_k - 1} \eta_t \phi'(Y_t) \bigg| \leq \sum_{t=e_{k-1}}^{e_k - 1} \eta_t |\phi'(Y_t)|  \leq \frac{L \beta_j^3}{4 R^2} H_j^{(k)}
        \leq \frac{5}{16} \kappa_0 \beta_j = 2 \beta_j.
    \]
    Since \( |Y_{e_{k-1}}|, |Y_{e_k}| \leq \beta_j \) on the event \( \mathcal{C} \), we have
    \[
        |G_k| \leq |Y_{e_k} - Y_{e_{k-1}}| + \bigg| \sum_{t=e_{k-1}}^{e_k - 1} \eta_t \phi'(Y_t) \bigg|
        \leq 4 \beta_j.
    \]
    Hence, \( \mathcal{C} \subset \bigcap_{k=1}^{K} \{ |G_k| \leq 4 \beta_j \} \).

    Finally, by Cauchy--Schwarz, we have
    \[
        V_k = \nu^2 \sum_{t=e_{k-1}}^{e_k - 1} \eta_t^2 \geq \frac{\nu^2 (H_j^{(k)})^2}{m_k} \geq
        \frac{\vartheta \sigma^2 \tau_j^2}{m_k}.
    \]
    Using the fact that \( \sum_{k=1}^{K} m_k \leq N_j \) and \( K \geq (2 H_j) / (5 \tau_j) \), we get
    \begin{align*}
        \frac{1}{K} \sum_{k=1}^{K} \frac{\beta_j^2}{V_k} &\leq \frac{\beta_j^2}{K} \sum_{k=1}^{K}
        \frac{m_k}{\vartheta \sigma^2 \tau_j^2} 
        \leq \frac{\beta_j^2 N_j}{\vartheta \sigma^2 K \tau_j^2} \\
        &\leq \frac{5 \tau_j \beta_j^2 N_j}{2 \vartheta \sigma^2  H_j \tau_j^2}
        = \frac{5 \beta_j^2 N_j}{2 \vartheta \sigma^2 H_j \tau_j} \\
        &= \frac{5 L N_j}{2 \vartheta \sigma^2 H_j} \cdot \frac{1}{\kappa_0 R^2} \cdot \beta_j^4
    \end{align*}
    Now use the upper bound in Lemma~\ref{lem:tolerance-bound} to conclude
    \[
        \frac{1}{K} \sum_{k=1}^{K} \frac{\beta_j^2}{V_k}
        \leq \frac{5 L N_j}{2 \vartheta \sigma^2 H_j} \cdot \frac{1}{\kappa_0 R^2}
        \cdot \frac{64 R^2 \bar{\varepsilon}_j}{L} = \frac{25}{\vartheta} B \frac{\hat{h}_j \sqrt{N_j}}
        {\sigma^2 H_j}.
    \]
\end{proof}

Now we are ready to prove Lemma~\ref{lem:dyadic-epoch-ceiling}.

\begin{proof}[Proof of Lemma~\ref{lem:dyadic-epoch-ceiling}]
    We use contradiction. Suppose
    \[
        H_j > C_b \frac{B}{\sigma^2} \frac{\hat{h}_j \sqrt{N_j}}{j}.
    \]
    Then by the definition of \( \tau_j \) and Lemma~\ref{lem:tolerance-bound}, we have
    \begin{align*}
        \frac{H_j}{\tau_j} &= H_j \cdot \frac{L \beta_j^2}{\kappa_0 R^2} \\
        &\geq C_b \frac{B}{\sigma^2} \frac{\hat{h}_j \sqrt{N_j}}{j} \cdot \frac{L \beta_j^2}{\kappa_0 R^2} \\
        &\geq C_b \frac{B}{\sigma^2} \frac{\hat{h}_j \sqrt{N_j}}{j} \cdot \frac{L}{\kappa_0 R^2}
        \cdot 4 R \sqrt{\frac{\bar{\varepsilon}_j}{L}} \\
        &= \frac{4 C_b B^{3/2}}{\kappa_0} \frac{L^{1/2}}{R \sigma^2} \frac{\hat{h}_j^{3/2} N_j^{1/4}}{j}
        \to \infty, \quad \text{as } j \to \infty.
    \end{align*}
    Then for sufficiently large \( j \), we have \( H_j / \tau_j \geq 2 \). Then 
    Lemma~\ref{lem:block-noise-constraint} applies, and we have
    \[
        \mathcal{C} \subset \bigcap_{k=1}^{K} \{ |G_k| \leq 4 \beta_j \},
    \]
    and
    \[
        \frac{1}{K} \sum_{k=1}^{K} \frac{(4 \beta_j)^2}{V_k} \leq \frac{400}{\vartheta} 
        B \frac{\hat{h}_j \sqrt{N_j}}{\sigma^2 H_j}.
    \]
    By Markov's inequality, this implies that there exists at least \( K/2 \) sub-blocks with
    \[
        \frac{(4 \beta_j)^2}{V_k} \leq \frac{800}{\vartheta} B \frac{\hat{h}_j \sqrt{N_j}}{\sigma^2 H_j}
        \leq \gamma_0 j, \quad j \in \mathcal{G}_j, |\mathcal{G}_j| \geq K/2.
    \]
    where the last inequality follows from the contradiction assumption and the definition of \( C_b \).

    Also, Lemma~\ref{lem:sub-block-algebra} gives \( K \geq (2 H_j) / (5 \tau_j) \). Therefore, we have
    \[
        K \geq \frac{C_b B^{3/2}}{4} \frac{L^{1/2}}{R \sigma^2} \frac{\hat{h}_j^{3/2} N_j^{1/4}}{j}.
    \]
    Since
    \[
        \mathcal{C} \subset \bigcap_{k=1}^{K} \{ |G_k| \leq 4 \beta_j \} \subset
        \bigcap_{k \in \mathcal{G}_j} \{ |G_k| \leq 4 \beta_j \},
    \]
    we have
    \begin{align*}
        \mathbb{P}(\mathcal{C}) &\leq \prod_{k \in \mathcal{G}_j} \mathbb{P}(|G_k| \leq 4 \beta_j) \\
        &= \prod_{k \in \mathcal{G}_j} \mathbb{P} \bigg( \frac{|G_k|}{\sqrt{V_k}} \leq 
        \frac{4 \beta_j}{\sqrt{V_k}} \bigg) \\
        &\leq \prod_{k \in \mathcal{G}_j} \mathbb{P} \bigg( |Z| \leq \sqrt{\gamma_0 j} \bigg) \\
        &= \prod_{k \in \mathcal{G}_j} (1 - 2 \Phi(-\sqrt{\gamma_0 j})) \\
        &\leq \exp(-|\mathcal{G}_j| \cdot 2 \Phi(-\sqrt{\gamma_0 j})) \\
        &\leq \exp\bigg( - K \cdot \frac{\sqrt{\gamma_0 j}}{1 + \gamma_0 j} \frac{1}{\sqrt{2} \pi}
        \mathrm{e}^{-\gamma_0 j/2} \bigg),
    \end{align*}
    where the last inequality uses the Mills lower bound
    \[
        \Phi(-x) \geq \frac{x}{1+x^2} \frac{1}{\sqrt{2\pi}} \mathrm{e}^{-x^2/2}, \quad x > 0.
    \]
    Now substituting the lower bound of \( K \) gives
    \[
        \mathbb{P}(\mathcal{C}) \leq \exp \bigg( -\frac{C_b B^{3/2}}{4 \sqrt{2} \pi} \frac{L^{1/2}}{R \sigma^2}
        \cdot \frac{\sqrt{\gamma_0} \hat{h}_j^{3/2} 2^{j/8}}{j^{1/2}(1 + \gamma_0 j)} \bigg) \to 0,
        \quad \text{as } j \to \infty.
    \]
    This contradicts the assumption that \( \mathbb{P}(\mathcal{C}) \geq 1 - \alpha \). Hence, we must have
    \[
        H_j \leq C_b \frac{B}{\sigma^2} \frac{\hat{h}_j \sqrt{N_j}}{j},
    \]
    for all sufficiently large \( j \).
\end{proof}

\subsection{Dyadic sequence estimates}
\label{appendixsub:dyadic-sequence-estimates}

Our main focus here is to prove Lemma~\ref{lem:step-upper-bound} and \ref{lem:energy-amplification}.

\begin{proof} [Proof of Lemma~\ref{lem:step-upper-bound}]
    Note that \( H_j = S_{N_{j+1}} - S_{N_j} \). By Lemma~\ref{lem:dyadic-epoch-ceiling}, there
    exists \( j_0 \in \mathbb{N}_{+} \) such that
    \[
        \frac{H_j}{\sqrt{N_j}} = \sqrt{2} p_{j+1} - p_j \leq C_b \frac{B}{\sigma^2} \frac{\hat{h}_j}{j},
        \quad \forall j \geq j_0.
    \]
    The above recursion implies that
    \[
        p_j \leq 2^{-(j-j_0)/2} p_{j_0} + \sum_{i=j_0}^{j-1} 2^{-(j-i)/2} C_b \frac{B}{\sigma^2} 
        \frac{\hat{h}_i}{i}.
    \]
    Note that for \(j > i\), we have \( \hat{h}_i \leq h_j \). Now we split the summation into two parts:
    \begin{align*}
        p_j &\leq 2^{-(j-j_0)/2} p_{j_0} + \sum_{i=j_0}^{\lfloor j/2 \rfloor - 1} 2^{-(j-i)/2} 
        C_b \frac{B}{\sigma^2} \frac{h_j}{i} + \sum_{i=\lfloor j/2 \rfloor}^{j-1} 2^{-(j-i)/2} C_b 
        \frac{B}{\sigma^2} \frac{h_j}{i} \\
        &\leq 2^{-(j-j_0)/2} p_{j_0} + \frac{j}{j_0} 2^{-j/4} C_b \frac{B}{\sigma^2} h_j + 
        C_b \frac{B}{\sigma^2} \frac{2 h_j}{j(1 - 2^{-1/2})}. 
    \end{align*}
    It is straightforward to verify that \( j 2^{-j/4} = o(1/j) \). Thus, the first two terms can be absorbed into
    the last term multiplied by a constant. Therefore, there exists a constant \( C_p \) such that
    \[
        p_j \leq C_p \frac{B}{\sigma^2} \frac{h_j}{j}.
    \]
    By Corollary~\ref{cor:step-lower-bound}, we have
    \[
        \mathcal{Q}_j \leq \frac{64 B}{z_{\alpha}^2 \sigma^2} p_j h_j
        \leq C_q \frac{B^2}{\sigma^4} \frac{h_j^2}{j},
    \]
    where \( C_q \) is a constant.
\end{proof}

\begin{proof} [Proof of Lemma~\ref{lem:energy-amplification}]
    Consider the recursion
    \[
        \mathcal{Q}_{j+1} - \mathcal{Q}_j = \sum_{t=N_j}^{N_{j+1}-1} \eta_t^2 \geq \frac{H_j^2}{N_j}.
    \]
    Note that \( H_j / \sqrt{N_j} = \sqrt{2} p_{j+1} - p_j \). Then we have
    \begin{align*}
        \mathcal{Q}_J - \mathcal{Q}_m &\geq \sum_{j=m}^{J-1} (\sqrt{2} p_{j+1} - p_j)^2 \\
        &= \sum_{j=m}^{J-1} (2 p_{j+1}^2 + p_j^2 - 2 \sqrt{2} p_j p_{j+1}) \\
        &\geq \sum_{j=m}^{J-1} ((2 - \sqrt{2}) p_{j+1}^2 + (1 - \sqrt{2}) p_j^2) \\
        &= (2 - \sqrt{2}) p_J^2 + (1 - \sqrt{2}) p_m^2 + (\sqrt{2} - 1)^2 \sum_{j=m+1}^{J-1}
        p_j^2 \\
        &\geq (1 - \sqrt{2}) p_m^2 + (\sqrt{2} - 1)^2 \sum_{j=m+1}^{J-1} p_j^2.
    \end{align*}
    Therefore, we have
    \[
        \mathcal{Q}_J \geq (2 - \sqrt{2}) \mathcal{Q}_m + (\sqrt{2} - 1)^2 \sum_{j=m+1}^{J-1} p_j^2.
    \]
    Then by Corollary~\ref{cor:step-lower-bound}, we have
    \[
        \mathcal{Q}_J \geq (2 - \sqrt{2}) \mathcal{Q}_m + (\sqrt{2} - 1)^2 \frac{z_{\alpha}^4 \sigma^4}
        {2^{12} B^2} \sum_{j=m+1}^{J-1} \frac{\mathcal{Q}_j^2}{h_j^2}.
    \]
\end{proof}